\documentclass[12pt,a4paper]{amsart}
\usepackage[top=1.15in, bottom=1.15in, left=1.17in, right=1.17in]{geometry}
\usepackage{amsthm, amsmath, amssymb, amscd, latexsym, multicol, verbatim, enumerate, graphicx,xy, color}
\usepackage{amstext,amsfonts,amsbsy}
\usepackage{tikz}
\usepackage[english]{babel}
\usepackage{cite}
\usepackage{stmaryrd}
\usepackage{longtable}
\usepackage{pdflscape}
\usepackage[colorlinks=true, pdfstartview=FitV, linkcolor=blue, citecolor=blue, urlcolor=blue, breaklinks=true]{hyperref}

\makeatother \theoremstyle{remark}
\numberwithin{equation}{section}

\theoremstyle{definition}
\newtheorem{definition}{Definition}\theoremstyle{definition}
\newtheorem{proposition}{Proposition}
\newtheorem{theorem}{Theorem}
\newtheorem{lemma}{Lemma}
\newtheorem{corollary}{Corollary}
\newtheorem{remark}{Remark}
\newtheorem{example}{Example}

\DeclareMathOperator{\Flag}{\mathcal{F}\hspace{-1.6pt}\ell}

\DeclareMathOperator{\init}{in}

\newcommand{\Gr}{{\operatorname*{Gr}}}

\title[Schubert varieties and Feigin's degeneration]{Following Schubert varieties under Feigin's degeneration of the flag variety}

\begin{document}

\author[L. Bossinger, M. Lanini]{Lara Bossinger, and  Martina Lanini}
\address{
Instituto de Matem\'aticas Unidad Oaxaca, 
Universidad Nacional Aut\'onoma de M\'exico,
Le\'on 2, altos, Oaxaca de Ju\'arez,
Centro Hist\'orico,
68000 Oaxaca,
Mexico}
\email{lara@im.unam.mx}
\address{Dipartimento di Matematica, Universit\`a di Roma ``Tor Vergata'',  Via della Ricerca Scientifica 1, I-00133 Rome, Italy}

\email{lanini@mat.uniroma2.it}

\maketitle

\begin{abstract}
We study the effect of Feigin's flat degeneration of the type $\textrm{A}$ flag variety on the defining ideals of its Schubert varieties. 
In particular, we describe two classes of Schubert varieties which stay irreducible under the degenerations and in several cases we are able to encode reducibility of the degenerations in terms of symmetric group combinatorics. 
As a side result, we obtain an identification of some {\it degenerate Schubert varieties} (i.e. the vanishing sets of initial ideals of the ideals of Schubert varieties with respect to Feigin's Gr\"obner degeneration) with Richardson varieties in higher rank partial flag varieties.
\end{abstract}
\section{Introduction}

Let $G$ be a complex simple Lie group and let $P\subset G$ be a parabolic subgroup. In \cite{F12}, Feigin introduced a flat degeneration of the flag variety $G/P$, which is equipped with an action of the $M$-fold product of the additive group of the field ($M$ being the dimension of a maximal unipotent subgroup of $G$).
 These degenerations of flag varieties (and some generalizations in type ${\rm A}$) have been in the past years intensively studied from many different perspectives (see, for example, \cite{Fei11}, \cite{CFR12}, \cite{C-IL15}, \cite{F16}, \cite{CFFFR17}, \cite{LS18}).
 
In this paper, we deal with the effect of Feigin's degeneration on the Schubert varieties inside $\Flag_n:=SL_n/B$, for $B$ the Borel subgroup of upper triangular matrices. 
In \cite{F12} it is shown that in type $\rm A$ the degeneration $\Flag_n^a$ of $\Flag_n$ can be embedded into the product of projective spaces, exactly as $\Flag_n$, and that the defining ideal is generated by degenerate Pl\"ucker relations. 
More precisely, the defining ideal $\mathcal{I}_{\Flag_n}$ of $\Flag_n$ is generated by Pl\"ucker relations and the defining ideal $\mathcal{I}_{\Flag_n}^a$ is obtained as the initial ideal $\init_{\textbf w}(\mathcal{I}_{\Flag_n})$ with respect to a weight vector $\mathbf{w}$ (whose components are indexed by Pl\"ucker coordinates), as described in Section \ref{sec:plücker}. 
On the other hand, if $v\in S_n$ is a permutation, it is well-known that the ideal $\mathcal{I}_v$ of the Schubert variety $X_v=\overline{BvB/B}\subseteq \Flag_n$ is generated by the Pl\"ucker relations together with certain Pl\"ucker coordinates (see \S\ref{sec:ideals} for a more precise formulation). 
Thus it is natural to ask what happens to $\mathcal{I}_v$ under Feigin's degeneration, that is to investigate $\init_{\textbf w}(\mathcal{I}_v)$.

From the first non-trivial example, it is already clear that not all Schubert varieties under Feigin's degeneration will stay irreducible: for $n=3$, indeed, one of the six Schubert varieties degenerates to a reducible variety. 
Therefore, a consistent part of this paper is directed towards understanding this reducibility phenomenon.  

We recall that the cohomology ring of $\Flag_n$ can be identified (after doubling the degree) with its Chow ring, and the latter is generated by Schubert classes. Moreover, it is shown in \cite{CFR12} that $\Flag_n^a$ admits an affine paving and hence also its cohomology ring can be identified (up to doubling the degree) with its Chow ring. Therefore we expect the above mentioned reducibility phenomenon to be related to the surjectivity of the cohomology ring map $\psi:H^*(\Flag_n^a, \mathbb{Z})\rightarrow H^*(\Flag_n, \mathbb{Z})$ proven in \cite{LS18}. It would be interesting to investigate this relationship, and in particular deduce a description of the kernel of $\psi$ in terms of Schubert classes.

We should mention here that what we refer to as \emph{Feigin's degeneration} is in fact a modified version of his original construction, which was coming from Lie theory. The version we deal with in this paper is the one which has been studied in \cite{C-IL15}. The variety one obtains in this way is isomorphic to Feigin's original degeneration, but in some sense it behaves better with respect to Schubert varieties. In fact, Caldero noticed in \cite{Cal02} that there does not exist a (flat) toric degeneration of the flag variety under which all Schubert varieties degenerate to toric varieties. For $n=3$ (which is the only case, apart from $n=2$, in which $\Flag_n^a$ is toric) our version of the degeneration preserves irreducibility of all but one Schubert variety, while two Schubert varieties would become reducible under Feigin's original definition. This is why we feel that in this setting the definition we use is sort of optimal.

Before focusing on Schubert varieties which become reducible after degenerating, we first describe some cases in which they stay irreducible (see Section \ref{sec:irred}). 
In particular, we prove that there is a class of Schubert varieties (indexed by permutations which are less or equal than a distinguished Coxeter element) whose  defining ideals are  not affected  by the degeneration (see Proposition \ref{prop:isomschubert}).

Section \ref{sec:reducibility} is devoted to sufficient conditions on the permutation $v$ such that the initial ideal $\init_{\bf w}(\mathcal{I}_v)$ is not prime. 
The strategy is to look for Pl\"ucker relations whose initial term is a (degree 2) monomial when considered modulo the Pl\"ucker coordinates which vanish on $X^a_v:=V(\init_{\bf w}(\mathcal I_v))$, which coincide with the ones vanishing on $X_v$. 
The efficiency of some of the conditions we give is then tested by looking at the $n=4$ and $n=5$ examples, for which we can detect all initial ideals containing monomials (see Tables \ref{tab:S4} and \ref{tab:S5}).

In previous joint work with Cerulli Irelli \cite{C-IL15}, the second author proved that the degenerate flag variety $\Flag^a_n$ can be embedded in the flag variety $SL_{2n-2}/P$ of partial flags in $\mathbb{C}^{2n-2}$ consisting of odd dimensional spaces (that is, $P=P_{\omega_1+\omega_3+\ldots\omega_{2n-3}}$). Under this embedding, it was shown in \cite{C-IL15} that $\Flag_n^a$ is isomorphic to a Schubert variety. From this fact  (together with classical results) one could obtain  a new proof of projective normality, Frobenius splitting,  and rationality of the singularities of $\Flag_n^a$. In Section \ref{sec:rich}
we further exploit such an isomorphism and study the effect of Feigin's degeneration on Schubert varieties inside $SL_{2n-2}/P$. 
The idea is to show irreducibility of the degeneration of some Schubert variety  by proving that the above-mentioned embedding sends it to a Richardson variety. 
Although our main focus is the analysis of Pl\"ucker relations (cf. Sections \ref{sec:reducibility} and \ref{sec:irred}), for which there is no need to move to a higher rank (partial) flag variety, we decided to dedicate a section to the connection  with  Richardson varieties.
By comparing Proposition \ref{prop:isomschubert} with Lemma \ref{lem:richardson} we obtain a realization of some Richardson varieties inside $SL_{2n-2}/P$ as Schubert varieties in a lower rank (complete) flag variety.

The last section of the paper deals with Schubert divisors, that is Schubert varieties of codimension one in $\Flag_n$. By applying our reducibility criteria from Section \ref{sec:reducibility}, we are able to prove that if $n$ is even all Schubert divisors become reducible, while for $n$ odd this happens for all but one. In this case, the remaining divisor is shown to be isomorphic to a Richardson variety inside $SL_{2n-2}/P$, and hence irreducible.

We want to point out that our paper is very different in spirit from \cite{F16}, where  irreducible flat degenerations of Schubert varieties corresponding to some special Weyl group elements (\emph{triangular elements}) are produced by considering PBW-degenerations of Demazure modules $V_w(\lambda)$ and then realizing the desired degeneration as the closure of an appropriate $\mathbb{G}_a^M$-orbit inside $\mathbb{P}(V_w(\lambda))$. So for any Schubert variety which is indexed by a triangular element (see \cite[Definition 1]{F16}) one can construct a flat irreducible degeneration via Fourier's procedure, while in this article we fix the degeneration (Feigin's) of the whole flag variety and study its effect on Schubert varieties (which are hence simultaneously degenerated). 

Since a first draft of this paper appeared on the arxiv more research has been done regarding degenerations of Schubert varieties. Among them \cite{CM}, where similar methods are employed to study Gr\"obner degenerations of Schubert varieties, and \cite{ChiriviFangFourier_degSchubert,Kambaso_demazure} which are very different in flavour and closely related to \cite{F16}. 

\medskip
\textbf{Acknowledgements.} 
Both authors would like to thank Sara Billey, Rocco Chiriv\`i, Xin Fang, Evgeny Feigin, Ghislain Fourier, Fatemeh Mohammadi and Markus Reineke for their comments on a preliminary version of this paper.
Most of this project was developed during a research visit of L.B. at Universit\`a di Roma ``Tor Vergata'' supported by QM$^2$ through the Institutional Strategy of the University of Cologne (ZUK 81/1). 
L.B. further acknowledges support of the PAPIIT project IA100122, Direcci\'on General de Asuntos del Personal Acad\'emico, Universidad Nacional Aut\'onoma de M\'exico 2022.  M. L. acknowledges the MIUR Excellence Department Project 2023--2027 awarded to the Department of Mathematics, University of Rome Tor Vergata, CUP E83C18000100006 and the 
PRIN2017 CUP E8419000480006.

\section{Preliminaries and notation}

\subsection{Symmetric group combinatorics}
The combinatorics of the symmetric group control many geometric properties of $\Flag_n$ and its Schubert varieties, therefore we spend a little time here introducing the notation we will need later on.

For any two positive integers $i,j\in\mathbb{Z}_ {\geq 1}$, with $i\leq j$ we denote by $[i,j]:=\{a\in\mathbb{Z}\mid i\leq a\leq j\}$. Moreover, we use the short hand notation  $[j]:=[1,j]$. We write $\binom{[n]}{k}$ for the set of subsets of cardinality $k$ inside $[n]$. 

Let $n\geq 2$ and denote by $S_n$ the symmetric group.
Recall that the symmetric group $S_n$ admits a presentation as a Coxeter group, with set of simple reflections $\{s_i\mid i=1, \ldots, n-1\}$, where $s_i$ denotes the transposition $(i,i+1)$. We will use the standard terminology and say that a product $s_{i_1}\ldots s_{i_r}$ is a reduced expression for $v\in S_n$ if $v=s_{i_1}\ldots s_{i_r}$ and all other expressions of $v$ as a product of simple reflections $v=s_{j_1}\ldots s_{j_t}$ are such that $t\geq r$. In this case $r=\ell(w)$ is called the \emph{length} of $w$. We denote by $\leq$ the Bruhat order on $S_n$ and recall the following equivalent characterization (see, for example, \cite[Theorem~2.1.5]{BB05}):
For $v\in S_n$ and $i,j\in [n]$ set
\begin{eqnarray}\label{eq:w[i,j]}
w^{i,j}=\#\{a\in [i]\mid w(a)\ge j\}.
\end{eqnarray}
Then
\begin{eqnarray}\label{eq:BBequiv}
v\le u \Leftrightarrow v^{i,j}\le u^{i,j} \text{ for all } i,j.
\end{eqnarray}

In the following we will also need that if $v\in S_n$ and $i\in[n-1]$, then 
\[ vs_i<v\quad \Leftrightarrow v(i)>v(i+1),\]
or, equivalently, 
\[ s_iv<v\quad \Leftrightarrow v^{-1}(i)>v^{-1}(i+1).\]
The symmetric group $S_n$ acts on $\binom{[n]}{k}$ for any $k$: if $I=\{i_1, \ldots, i_k\}\in\binom{[n]}{k}$ then
\[
v(I):=\{v(i_1), \ldots v(i_k)\}.
\]
This action is transitive, so that $\binom{[n]}{k}$ is identified with the $S_n$-orbit of $[k]$ and hence with the set of minimal length coset representatives in $S_n/\langle s_1, \ldots, s_{k-1}, s_{k+1},\ldots, s_{n-1}\rangle$. In this way the Bruhat order on $S_n$ induces a partial order on $\binom{[n]}{k}$ (see, for instance, \cite[Proposition 2.5.1]{BB05}) that we also denote by $\leq$.
We will sometimes write elements $v\in S_n$ as $[v(1),v(2),\dots,v(n)]$. This is referred to as the \emph{one-line} notation. 

\subsubsection{Sequences} In the following sections, we will often need to deal with sequences $(i_1, \ldots, i_k)$ rather than sets $\{i_1, \ldots, i_k\}$. We denote by $\mathcal{S}(n,k)$ the set of sequences of $k$ pairwise distinct numbers between $1$ and $n$.

Given two sequences $I_1=(i_1^{(1)}, \ldots, i_{k}^{(1)})\in \mathcal{S}(n,k)$, $I_2=(i_1^{(2)}, \ldots, i_l^{(2)})\in\mathcal{S}(n,l)$ such that $I_1\cap I_2=\emptyset$, we denote by $(I_1,I_2):=(i_1^{(1)}, \ldots, i_k^{(1)}, i_1^{(2)}, \ldots i_l^{(2)})\in \mathcal{S}(n,k+l)$ the sequence obtained by concatenation.

If $L\in\mathcal{S}(n,d)$ and $J=(j_1, \ldots, j_k)\in\mathcal{S}(n,k)$, then the sequence $L'=(L\setminus(l_{r_1}, \ldots, l_{r_k}))\cup (j_1, \ldots j_k)\in\mathcal{S}(n,d)$ is obtained from $L$ by substituting the subsequence $(l_{r_1}, \ldots, l_{r_k})$ with $(j_1, \ldots, j_k)$, that is $l'_a=l_a$ if $a\not\in\{r_1, \ldots, r_k\}$ and $l'_a=j_b$ if $a=r_b$.
There is a forgetful map
\[
F:\mathcal{S}(n,k)\rightarrow \binom{[n]}{k}, \quad (i_1, \ldots, i_k)\mapsto \{i_1, \ldots i_k\}.
\]
By abuse of notation, if $I\in\mathcal{S}(n,k)$ and $v\in S_n$, we will write $I\leq v([k])$ instead of $F(I)\leq v([k])$ (and $I\geq v([k])$, $I\nleq v([k])$, etc.,  will have an analogous meaning). 

\subsubsection{A special Coxeter element}
The Coxeter element $c=s_{n-1}s_{n-2}\cdots s_2s_1\in S_n$ will play an important role later on. Observe that in the one-line notation
\[
c=[n,1,2,3\ldots, n-1].
\] 
Thus, by \cite[Proposition 2.4.8]{BB05}, 
for a subset $I\in \binom{[n]}{d}$ the following holds
\begin{eqnarray}\label{eq:coxeter}
I\le c([d])\Leftrightarrow I=[d-1]\cup \{b\}\text{ for }d\le b\le n.
\end{eqnarray}

\subsection{Basics on the flag variety}
Let $n\geq 2$. We denote by $\Flag_n$ the variety of complete flags in $\mathbb{C}^n$.  Let $(e_i)_{1\le i\le n}$ be the standard basis of $\mathbb C^n$. Let $B\subset SL_n$ be the Borel subgroup of upper triangular matrices.  The group $SL_n$ acts (by base change) transitively on $\Flag_n$ and we can identify the flag variety with the quotient $SL_n/B$ by  looking at the $SL_n$-orbit of the standard flag $E_\bullet=(\{0\}\subset E_1\subset \dots\subset E_{n-1}\subset \mathbb C^n)\in \Flag_n$ with
\[ E_i:=\textrm{span}_{\mathbb{C}}\{e_1, \ldots, e_i\} \quad (i=1, \ldots n-1).\]
  Recall that under the left action of $B$, the flag variety decomposes as a union of Schubert cells indexed by the elements of the symmetric group $S_n$:
\[
SL_n/B= \bigsqcup_{v\in S_n} BvB/B
\]
where, by abuse of notation, $v$ in $BvB/B$ denotes the corresponding permutation matrix in $SL_n$. Finally, let $X_v$ be the Schubert variety, that is the closure  $\overline{BvB/B}$ of a Schubert cell.

Analogously, also $B_-$, the Borel subgroup of lower triangular matrices, acts by left multiplication on $SL_n/B$, providing the decomposition:
\[
SL_n/B= \bigsqcup_{u\in S_n} B_-uB/B.
\]
We denote by $X^u$ the opposite Schubert variety $\overline{B_-uB/B}$. In \S\ref{sec:rich}, we will also consider Richardson varieties $X_v^u:=X_v\cap X^u$.

\subsubsection{Pl\"ucker relations}\label{sec:plücker}
Our main reference for Pl\"ucker coordinates and relations is \cite{Fu97}, while we refer to \cite{F12} for the degenerate Pl\"ucker relations.

We start by recalling the Pl\"ucker embedding of a Grassmannian. Recall that $(e_i)_{1\le i\le n}$ is the standard basis of $\mathbb C^n$, so that 
\[
\{e_{i_1}\wedge\dots\wedge e_{i_k}\mid 1\leq i_1<i_2<\ldots <i_k\leq n\}
\]
is a basis of $\wedge^k\mathbb C^n$. Let $(\wedge^k\mathbb C^n)^*$ be the dual vector space, then the Pl\"ucker coordinate $p_{i_1,\dots,i_k}\in(\wedge^k\mathbb C^n)^*$ for $1\leq i_1<i_2<\ldots <i_k\leq n$ is defined to be the basis element dual to $e_{i_1}\wedge\ldots \wedge e_{i_k}$.
For  $i_1, \ldots ,i_k\in [n]$ pairwise distinct, but not necessarily increasing, the Pl\"ucker coordinate $p_{i_1, \ldots, i_k}$ has the following property
\[
p_{\sigma(i_1),\dots,\sigma(i_k)}=(-1)^{\ell(\sigma)}p_{i_1,\dots,i_k} \quad  \hbox{ for all }\sigma \in S_n.
\]
Denote by $p_I$ the Pl\"ucker coordinate corresponding to a sequence $I=(i_1,\dots,i_k)\in\mathcal{S}(n,k)$. In the following sections it will be sometimes convenient to simplify notation and index some Pl\"ucker coordinates by a set instead of a sequence. This has to be interpreted as being indexed by the sequence obtained by arranging the elements of the set in an increasing order.

We have obtained in this way the Pl\"ucker embedding
\begin{equation}\label{eq:plückerembed}
\Gr(k,\mathbb C^n)\hookrightarrow \mathbb P(\wedge^k\mathbb C^n).
\end{equation}
The flag variety is embedded in the product of Grassmannians
\begin{eqnarray*}
\Flag_n \hookrightarrow \Gr(1,\mathbb C^n)\times \Gr(2,\mathbb C^n)\times\dots\times\Gr(n-1,\mathbb C^n).
\end{eqnarray*}
By composing the latter embedding with the embedding \eqref{eq:plückerembed} for each Grassmannian in the product, we get
\[
\Flag_n \hookrightarrow \mathbb{PC}^n\times\mathbb P(\wedge^2\mathbb C^n)\times\dots\times\mathbb P(\wedge^{n-1}\mathbb C^n).
\]
Denote by $\mathcal{I}_{\Flag_n}$ the (homogeneous) ideal of $\Flag_n$ in $\mathbb{C}[p_{i_1, \ldots, i_k}\mid 1\leq i_1<i_2<\ldots<i_k\leq n, \ k\in[n-1]]$ with respect to this embedding. 
Then $\mathcal{I}_{\Flag_n}$ is generated by elements in
\[
\{R^k_{(j_1, \ldots, j_e), (l_1, \ldots, l_d)}\mid  e\leq d,  \ k\in [e] \}  
\]
given by
\begin{equation}\label{eq:plücker rel}
R^k_{J,L}=p_Jp_L-\sum_{1\le r_1<\dots<r_k\le d} p_{J'}p_{L'},
\end{equation}
  where $L=(l_1,\dots,l_d)\in\mathcal{S}(n,d)$, $J=(j_1,\dots,j_e)\in\mathcal{S}(n,e)$,
$L'=(L\setminus (l_{r_1},\dots,l_{r_k}))\cup (j_1,\dots, j_k)$ and
$J'= (J\setminus ( j_1,\dots,j_k))\cup(l_{r_1},\dots,l_{r_k})$. 
The elements $R_{J,L}^k$ will be referred to as Pl\"ucker relations.
To simplify notation we set 
\begin{eqnarray}\label{eq:mathcalL}
\mathcal{L}_{J,L}^k=\left\{
(J',L')\mid 
\begin{smallmatrix}
\exists 1\le r_1<\dots<r_k\le \#L,\\
J'= (J\setminus ( j_1,\dots,j_k))\cup(l_{r_1},\dots,l_{r_k})  ,\\
L'=(L\setminus (l_{r_1},\dots,l_{r_k}))\cup (j_1,\dots, j_k)
\end{smallmatrix}
\right\}.
\end{eqnarray}

The weight vector ${\bf w}\in \mathbb R^{\binom{n}{1}+\dots +\binom{n}{n-1}}$ is defined componentwise by setting for  $I=\{i_1,\dots,i_k\}\in\binom{[n]}{k}$
\[
{\bf w}_{I}=\#\{r\mid k\le i_r\le n-1\}.
\]
If $I_1, \ldots, I_r\in{[n]\choose k}$, the ${\bf w}$-weight of the monomial $\prod_{t=1}^r p_{I_t}$ is $\sum_{t=1}^r {\bf w}_{I_t}$, while the initial form of a polynomial $f$ consists of the sum of those monomials whose ${\bf w}$-weight is \emph{minimal} among the weights of all monomials in $f$. 
Given an ideal $\mathcal I\subset \mathbb C[p_{i_1,\dots,i_d}\mid 1\le i_1<i_2<\dots <i_d\le n,d\in [n-1]]$ its \emph{initial ideal} is $\init_{\bf w}(\mathcal I)=(\init_{\bf w}(f)\mid f\in \mathcal I)$.
A (finite) set of elements in $\mathcal I$ whose initial forms generate $\init_{\bf w}(\mathcal I)$ is called a \emph{Gr\"obner basis}.
In \cite[Theorem 3.13]{F12} computes a Gr\"obner basis  for $\mathcal I_{\Flag_n}$ whose elements are the Pl\"ucker relations \eqref{eq:plücker rel}. Their initial forms are given by
\begin{eqnarray*}
\init_{\bf w}(R^k_{J,L})=p_Jp_L-\sum_{
\begin{smallmatrix}
(J',L')\in \mathcal{L}_{J,L}^k\\
\{l_{r_1},\dots,l_{r_k}\}\cap[e,d-1]=\emptyset
\end{smallmatrix}}
 p_{J'}p_{L'},
\end{eqnarray*}
where the leading term is non-zero, only if 
\begin{eqnarray}\label{eq:degflagplückerlead}
\{j_1,\dots,j_k\}\cap[e,d-1]&=&\emptyset.
\end{eqnarray}
We can choose $J,L$ in such a way that \eqref{eq:degflagplückerlead} holds. Observe that for $e=d$, we always have $\init_{\bf w}(R^k_{J,L})=R_{J,L}^k$ since the condition \eqref{eq:degflagplückerlead} is empty.

\begin{definition}[\cite{F12}]
The \emph{degenerate flag variety} is  the vanishing of the ideal $\init_{\bf w}(\mathcal{I}_{\Flag_n})$, that is
\[
\Flag_n^a:=V(\init_{\bf w}(\mathcal{I}_{\Flag_n}))\subset  \mathbb{PC}^n\times\mathbb P(\wedge^2\mathbb C^n)\times\dots\times\mathbb P(\wedge^{n-1}\mathbb C^n).
\]
\end{definition}
\begin{remark} Feigin's original definition, valid for any simple Lie group, was different from the one we have just given, which is a characterization of the type $\textrm A$ degenerate flag variety by \cite[Theorem 3.13]{F12}. As already mentioned in the introduction, we modify Feigin's definition to match the one considered in \cite{C-IL15}. Explicitly, to obtain our degeneration from Feigin's original one, a global shift by $-1$ (modulo $n$) to all indices is needed.    
\end{remark}

\subsection{Ideals for Schubert varieties and their degeneration}\label{sec:ideals}


For $v\in S_n$ the defining ideal of the Schubert variety $X_v\subset \Flag_n$ is given by the vanishing of $(p_I)_{I\not\le v([\#I])}$. It is shown in \cite[\S10.12]{LLM} (see also \cite[Theorem 3]{KR}) that by  embedding $X_v\hookrightarrow  \mathbb{PC}^n\times\mathbb P(\wedge^2\mathbb C^n)\times\dots\times\mathbb P(\wedge^{n-1}\mathbb C^n)$, we obtain the ideal
\begin{equation}\label{eq:idealSchubert}
\mathcal{I}_v:=\mathcal{I}_{\Flag_n}+(p_I)_{I\not\le v([\#I])}
\end{equation}
of $\mathbb{C}[p_{i_1, \ldots, i_d}\mid 1\leq i_1<i_2<\ldots<i_d\leq n, \ d\in[n-1]]$.
Feigin's degeneration of the flag variety induces a degeneration $X^a_v\subset\Flag_n^a$ of any Schubert variety $X_v\subset \Flag_n$:
\begin{equation}
X_v^a:=V(\init_{\bf w}(\mathcal{I}_v))\subset  \mathbb{PC}^n\times\mathbb P(\wedge^2\mathbb C^n)\times\dots\times\mathbb P(\wedge^{n-1}\mathbb C^n).
\end{equation}

In what follows we study the initial ideals $\init_{\bf w}(\mathcal I_v)$ in detail. 
Note that $\init_{\bf w}(p_I)=p_I$ for all $I\in \mathcal{S}(n,d)$, for all $d\in[n-1]$. 
Moreover, we have an inclusion:
\begin{eqnarray}\label{eq:containment}
\init_{\bf w}(\mathcal{I}_v)\supseteq
(\init_{\bf w}(R^k_{J,L}))_{k,J,L}+(p_I)_{I\not\le v([\#I])}.
\end{eqnarray}

The following example shows that this inclusion may be strict. In the proof of Theorem~\ref{thm:small flag} instead we will encounter examples of \eqref{eq:containment} being an equality.

\begin{example}
Consider the ideal $\mathcal I_{\Flag_4}$. Among its Pl\"ucker relations we have
\[
p_4p_{123}-p_3p_{124}+p_2p_{134}-p_1p_{234}.
\]
The first two terms have ${\bf w}$-weight 2 while the last to have ${\bf w}$-weight 3, so its initial form is $p_4p_{123}-p_3p_{124}$.
Now consider 
$v=s_1s_2s_3\in S_4$, which in the  one-line notation is 
$[2,3,4,1]$. Hence, $\{p_I\}_{I\not \le v([\#I])}
=\{p_3,p_4,p_{14},p_{24},p_{34}\}$. 
In particular, this implies that $f:=p_2p_{134}-p_1p_{234}\in \mathcal I_v$ and by definition its initial form lies in $\init_{\bf w}(\mathcal I_v)$.
As both monomials have the same ${\bf w}$-weight, $f$ agrees with its initial form. Notice however that $f$ does not lie in $(\init_{\bf w}(R^k_{J,L}))_{k,J,L}+(p_I)_{I\not\le v([\#I])}$. This demonstrates that the containment in \eqref{eq:containment} is strict in general.
\end{example}

\section{Two classes of irreducible $X_v^a$}\label{sec:irred}

We investigate two classes of Schubert varieties which degenerate to irreducible varieties. In this section we use some basics on Gr\"obner bases which we summarize for completeness. For more details we refer to \cite{HH,Stu96}.

\medskip
A {\it term order} on $\mathbb C[x_1, \dots , x_n]$ is a total order $<$ on the set of monic monomials in $\mathbb C[x_1, \dots , x_n]$ such that for every $\alpha,\beta,\gamma$ in $\mathbb Z_{\ge 0}^n$ we have that 
\begin{center}
(i) $1\le {\bf x}^\alpha$, \quad  and \quad (ii) if ${\bf x}^\alpha<{\bf x}^\beta$, then ${\bf x}^{\alpha+\gamma}<{\bf x}^{\beta+\gamma}$. 
\end{center}
The {\it initial monomial} of an element $f=\sum_{\alpha\in \mathbb{Z}^n_{\geq 0}} c_{\alpha} {\bf x}^{\alpha}\in \mathbb C[x_1, \dots , x_n]$ with respect to $<$ is $\init_<(f):=\max_<\{{\bf x}^\alpha\mid c_\alpha\not=0\}$.
The {\it initial ideal} of an ideal $J\subseteq \mathbb C[x_1, \dots , x_n]$ with respect to $<$ is defined as $\init_<(J):=(\init_<(f)\mid f\in J)$.\medskip

Let $\init_<(J)$ be a monomial initial ideal of the ideal $J$ for some term order $<$ on $\mathbb{C}[x_1,\ldots,x_n]$. 
Then the set $\mathbb B_{<}:=\{\bar{\bf x}^\alpha \mid {\bf x}^\alpha\not \in \init_<(J)\}$ is a vector space basis of $\mathbb C[x_1, \dots , x_n]/J$ (and $\mathbb C[x_1, \dots , x_n]/\init_<(J)$) called {\it standard monomial basis}, see e.g. \cite[Proposition 1.1]{Stu96}.

\medskip
Let $<$ be a term order on $\mathbb C[x_1, \dots , x_n]$ and $\mathcal{G}=\{g_1,\ldots,g_s\}$ a finite generating set of an ideal $J$. 
Then the {\it $S$-polynomial} of $g_i$ and $g_j$ is defined as 
\begin{displaymath}
S(g_i,g_j):=\frac{{\rm lcm}(\init_{<}(g_i),\init_{<}(g_j))}{\init_{<}(g_i)}g_i-\frac{{\rm lcm}(\init_{<}(g_i),\init_{<}(g_j))}{\init_{<}(g_j)}g_j. 
\end{displaymath}
{\it Buchberger's criterion} says that $\mathcal G$ is a Gr\"obner basis if and only if for all $1\le i<j\le s$ the $S$-polynomial $S(g_i,g_j)$ reduces to zero with respect to $\{g_1,\dots,g_s\}$,  see e.g \cite[Theorem 2.3.2]{HH}.

\subsection{Small flag varieties}
The main result of this section is the following.

\begin{theorem}\label{thm:small flag}
Let $v\in S_n$ be the minimal representative of the longest word in $S_n/\langle s_{1},\dots,s_{i},s_{i+r},\dots,s_{n-1}\rangle$ for suitable $i$ and $r$.
Then
\[
X_v^a\cong \Flag_r^a.
\]
In particular, $X_v^a$ is irreducible.
\end{theorem}

Before we prove the result, let us establish some useful lemmata. Note that written in one-line notation $v$ is of form 
\[
v=[
1, 2, \dots , i, i+r, i+r-1,\dots, i+1, i+r+1, \dots, n
].
\]
So $v(j)=j$ for $j\in[i]\cup[i+r+1,n]$ and $v(i+k)=i+r-k+1$ for $k\in[r]$. 

\begin{lemma}
  For the Schubert variety we have {$X_v\cong \Flag_r$}. In particular, the only non-vanishing Pl\"ucker coordinates besides $p_{[s]}$ for $s\le n-1$ are associated with the index sets in
\begin{equation}\label{eq:nonzero Plueck}
    \mathcal J_v=\{I\mid I=[i]\cup\{l_1,\dots,l_s\},s\in[r-1], \ l_j\in [i+1,i+r]\ \forall j\}.
\end{equation}  
\end{lemma}

\begin{proof}
  There is a bijection 
\[
\rho:\mathcal{J}_v\rightarrow \bigcup_{s=1}^{r-1}{\mathcal[r]\choose s},\quad I\mapsto \tilde{I},
\]
where if $I=[i]\cup\{l_1,l_2\dots,l_s\}$, we set $\tilde{I}=\{l_1-i,l_2-i,\dots,l_s-i\}$. This induces a bijection between the set of Pl\"ucker coordinates $\neq p_{[s]}$, $s\in[n-1]\setminus [i+1,i+r]$, which are non-vanishing on $X_v$ (that is, the ones involved in the relevant Pl\"ucker relations) and  Pl\"ucker coordinates $(\tilde{p}_{K})$ which generate the coordinate ring of $\Flag_r$. 
Notice that for  $J,L$ with $F(J),F(L)\in\mathcal{J}_v$, the Pl\"ucker relation $R_{J,L}^k$ is not identically 0 if and only if $R_{\tilde{J},\tilde{L}}^k$ is not identically 0 (since this happens for $k\in[\#(L\setminus (L\cap J))]=[\#(\tilde{L}\setminus(\tilde{L}\cap\tilde{J}))]$).
In particular, $\mathcal I_v$ is generated by $\{R^k_{J,L}\}_{k,J,L\in \mathcal J_v}\cup\{p_I\}_{I\not \in \mathcal J_v \cup\{[s]\mid s\in [n-1]\}}$.  
We extend the bijection to a map
\begin{equation}\label{eq:def rho}
\rho:\mathbb C[p_I\mid I\subset [n]]\to \mathbb C[p_{\tilde I}\mid \tilde I\subset [r]], \quad p_I\mapsto \left\{\begin{matrix} p_{\tilde I} & \text{ if } I\in \mathcal J_v \\ 0 & \text{ otherwise }. \end{matrix}\right.
\end{equation}
Then $\rho( (R^k_{J,L})_{k;J,L\in\mathcal J_v})=\mathcal I_{\Flag_r}$.
\end{proof}

Next, we establish a connection between the defining ideal of the degenerate flag variety $\Flag_r^a$ and the initial ideal defining $X_v^a$. We keep the notation introduced in \eqref{eq:def rho} and \eqref{eq:nonzero Plueck}.

\begin{lemma}\label{lem:init tilde w}
    Let $\tilde {\bf w}$ be the weight for $\Flag_r$, then $\rho( \init_{\bf w}((R^k_{J,L})_{k,J,L\in\mathcal J_v}))= \init_{\tilde {\bf w}}(\mathcal I_{\Flag_r})$.
\end{lemma}

\begin{proof}
Let $L=((1,\dots,i),(l_1,\dots,l_d))>J=((j_1,\dots,j_e),(1,\dots,i))$. 
Consider the relation $R^k_{\tilde{J},\tilde{L}}$. Without loss of generality we can assume that $J$ and $L$ are chosen in such a way that $\init_{\bf w}(R^k_{J,L})$ contains the monomial $p_Jp_L$. All other monomials $p_{J'}p_{L'}$ in $\init_{\bf w}(R^k_{J,L})$ are obtained from $p_Jp_L$ by choosing $1\le r_1<\dots<r_k\le i+d$, such that $\{l_{r_1},\dots,l_{r_k}\}\cap [i+e,i+d-1]=\emptyset$. 
This is the case if and only if $\{\tilde{l}_{r_1},\dots,\tilde{l}_{r_k}\}\cap [e,d-1]=\emptyset$. 
\end{proof}

In what follows we use Feigin's standard monomial basis given by semistandard PBW-tableaux. 
As we work throughout the paper with conventions for the weight vector ${\bf w}$ as in \cite{C-IL15} a global shift in the indices of all Pl\"ucker variables is needed before we can use Feigin's basis in our setting. 
Whenever we use the combinatorics from \cite{F12} in this section we assume we have applied the global shift to our index sets.

Recall that by \cite[Theorem 4.10]{F12} there exists a standard monomial basis (indexed by \emph{semistandard PBW-tableaux}) for $\mathbb C[p_{\tilde I}\mid \tilde I\subset [r]]/\init_{\tilde{\bf w}}(\mathcal I_{\Flag_r})$ (and $\mathbb C[p_{\tilde I}\mid \tilde I\subset [r]]/\mathcal I_{\Flag_r}$), denote it by $\mathbb B_{\rm PBW}$. 

\begin{lemma}\label{lem:term order}
   There exists a term order $\prec$ on $\mathbb C[p_{\tilde I}\mid \tilde I\subset [r]]$ such that $\mathbb B_{\rm PBW}$ equals the standard monomial basis given by monomials not contained in $\init_{\prec}(\mathcal I_{\Flag_r})$. Moreover, the set $\{R^k_{\tilde J,\tilde L}\}$ is a Gr\"obner basis for $\mathcal I_{\Flag_r}$ with respect to $\prec$. 
\end{lemma}

\begin{proof}
In \cite[Lemma 4.9]{F12} Feigin introduces a partial order $\le$ on ${\mathbb C}[p_{\tilde I}\mid \tilde I\subset [r]]$ such that for every monomial ${\bf p}^{a_T}\in {\mathbb C}[p_{\tilde I}\mid \tilde I\subset [r]]$ corresponding to a non-semistandard PBW-tableau $T$ there exists an element $f\in \init_{\tilde{\bf w}}(\mathcal I_{\Flag_r})$ that contains ${\bf p}^{a_T}$ in its support and further satisfies
\[
{\bf p}^{a_T} \ge {\bf p}^a \text{ for all } {\bf p}^a \text{ non zero monomial in } f.
\]
Moreover, $f={\bf p}^vR^k_{\tilde J,\tilde L}$ for a fixed monomial ${\bf p}^v$ that divides ${\bf p}^{a_T}$ and certain $k,\tilde J,\tilde L$.
Given $\tilde {\bf w}$ and the partial order $\le$ we define a term order on $\mathbb{C}[p_{\tilde I}:\tilde I\subset [r]]$ as follows: ${\bf p}^u\prec {\bf p}^v$ if and only if 
\begin{enumerate}
    \item $\tilde{\bf w}\cdot u >\tilde{\bf w}\cdot v$, or \footnote{Note the switch here: ${\bf p}^u\prec {\bf p}^v$ if $\tilde{\bf w}\cdot u >\tilde{\bf w}\cdot v$. This is because we have chosen to use the minimum convention for initial ideals with respect to weight vectors while for initial ideals with respect to term orders the maximum is considered.}
    \item $\tilde{\bf w}\cdot u =\tilde{\bf w}\cdot v$, and ${\bf p}^u\le {\bf p}^v$, or
    \item $\tilde{\bf w}\cdot u =\tilde{\bf w}\cdot v$, ${\bf p}^u$ and ${\bf p}^v$ are not comparable with respect to $\le$, and  ${\bf p}^u<_{\rm lex} {\bf p}^v$.
\end{enumerate}
Here $<_{\rm lex}$ denotes the lexicographic order on $\mathbb{C}[p_{\tilde I}:\tilde I\subset [r]]$ with underlying lexicographic order on the variables corresponding to their index sets.
Our term order $\prec$ is a refined version of a term order induced by a weight (see, for example the order $\prec_w$ in \cite[page 4]{Stu96}).
In particular, \cite[Proposition 1.8]{Stu96} holds also in our case and we have
\begin{eqnarray}\label{eq:init of init}
\init_{\prec}(\init_{\tilde{\bf w}}(\mathcal I_{\Flag_r})) = \init_{\prec}(\mathcal I_{\Flag_r}).
\end{eqnarray}
From \cite[Proof of Lemma 4.9]{F12} it follows that for ${\bf p}^{a_T}$ and $f$ as above we have
\begin{eqnarray}\label{eq:init non sstd}
\init_{\prec}(f)={\bf p}^{a_T}\in \init_{\prec}(\mathcal I_{\Flag_r}).
\end{eqnarray}
In particular, the cosets of the \emph{standard monomials}, i.e. ${\bf p}^u\not\in \init_{\prec}(\mathcal I_{\Flag_r})$, form a (standard monomial) basis for $\mathbb C[p_{\tilde I}:\tilde I\subset [r]]/\mathcal I_{\Flag_r}$.  
By \eqref{eq:init of init} they also form a basis for $\mathbb C[p_{\tilde I}:\tilde I\subset [r]]/\init_{\tilde{\bf w}}(\mathcal I_{\Flag_r})$), denote it by $\mathbb B_{\prec}$.
In particular, we deduce from \eqref{eq:init non sstd} that every standard monomial corresponds to a semistandard PBW-tableaux.
Hence, $\mathbb B_\prec \subset \mathbb B_{\rm PBW}$. But as both are bases for the same algebra they have to be equal. 
This implies the first claim. The second follows as $f={\bf p}^vR^k_{\tilde J,\tilde L}$, and so in particular ${\bf p}^{a_T}\in (\init_{\prec}(R^k_{\tilde J,\tilde L}))_{k,\tilde J,\tilde L}$.  
\end{proof}

\begin{proposition}\label{prop:gb}
    The set $\{R^k_{J,L}\}_{k,J,L\in \mathcal J_v}\cup \{p_I\}_{I\not \in \mathcal J_v}$ is a Gr\"obner basis for $\mathcal I_v$ and ${\bf w}$, denoted by $\mathcal G_{v;{\bf w}}$.
\end{proposition}

\begin{proof}
We use $\prec$ as defined in the proof of Lemma~\ref{lem:term order} and the map $\rho$ from \eqref{eq:def rho} to define a term order on $\mathbb C[p_{I}:I\subset [n]]$:
\[
{\bf p}^u<{\bf p}^t \quad \Leftrightarrow \quad  {\bf p}^u\not\in (p_I)_{I\not \in \mathcal J_v}\ni {\bf p}^t, \ \ \text{or}\ \ {\bf p}^t,{\bf p}^u\not\in (p_I)_{I\not \in \mathcal J_v} \ \ \text{and} \ \ \rho({\bf p}^u)\prec \rho({\bf p}^t).
\]
By definition of $<$ and Lemma~\ref{lem:init tilde w} we have $\init_{<}(\init_{\bf w}(R^k_{J,L})_{k;J,L\in \mathcal J_v})=\init_{<}((R^k_{J,L})_{k;J,L\in\mathcal J_v})$.
Moreover, as the $R^k_{\tilde J,\tilde L}$ constitute a Gr\"obner basis for $\mathcal I_{\Flag_r}$ and $\prec$ by Lemma~\ref{lem:term order}, it follows from Buchberger's criterion that the S-polynomials of pairs of these elements reduce to zero. 
Given the map $\rho$, the same must be true for S-polynomials of elements $R^k_{J,L}$ with $J,L\in \mathcal J_v$ with respect to the term order $<$.
Hence, in order to verify the claim we only need to compute S-polynomials of the relevant Pl\"ucker relations and the vanishing Pl\"ucker variables. 
Consider $R^k_{J,L}$ with $J,L\in\mathcal J_v$ and $p_I$ with $I\not \in \mathcal J_v$. Then $\init_{<}(R^k_{J,L})$ and $\init_<(p_I)$ are relatively prime. So by \cite[Lemma 2.3.1]{HH} their S-polynomials reduces to zero over $R^k_{J,L}$ and $p_I$. 
As the same is true for the S-polynomials of variables $p_I,p_{I'}$ with $I,I'\not\in \mathcal J_v$, the claim follows by Buchberger's criterion.  
\end{proof}

\begin{proof}[Proof of Theorem~\ref{thm:small flag}]
    We need to show that the isomorphism of $X_v$ and $\Flag_r$ induced by $\rho$ survives the degeneration. This is true as by Lemma~\ref{lem:init tilde w} and Proposition~\ref{prop:gb} $\rho$ maps the initial ideal defining $X_v^a$ to the ideal defining $\Flag_r$. 
    Lastly, by \cite[\S 5.1]{F12} the degenerate flag variety is the closure of a homogeneous space and therefore irreducible. As $X^a_v\cong \Flag_r^a$ by the above, the claim follows.
\end{proof}

Let  $\underline{i}=\{i_1,\dots,i_r\}\subsetneq[n-1]$. We set $m:=\min\{\underline{i}\}$, $M:=\max\{\underline{i}\}$, and $r:=M-m+1$. Let $v\in\langle s_{i_1},\cdots ,s_{i_r}\rangle\subset S_n$ and denote by $\widetilde{v}$ the element $\widetilde{s}_{i_1-m+1}\cdots \widetilde{s}_{i_r-m+1}\in S_r$. In this notation, from the proof of Theorem~\ref{thm:small flag} we can deduce the following result, which in this case allows one to reduce to smaller rank flag varieties.

\begin{corollary}\label{cor:smallflag}
Let $\underline{i}=\{i_1,\dots,i_r\}\subsetneq[n]$ and $v\in\langle s_{i_1},\cdots, s_{i_r}\rangle\subset S_n$. Then for $X^a_v\subset \Flag_n^a$ we have
\[
X^a_v\cong X^a_{\widetilde{v}}\subset \Flag_r^a.
\]
\end{corollary}

\subsection{Isomorphic degenerate and original Schubert varieties}\label{sec:deg-vs-orig}
 In the following we present another instance in which a Schubert variety stays irreducible under Feigin's degeneration of $\Flag_n$. In fact, for the class of varieties we deal with in this section a stronger property holds: the degeneration process does not deform them, that is $X^a_v$ is isomorphic to the original Schubert variety $X_v$.
 
Recall that we denote by $c\in S_n$ the special Coxeter element $c=s_{n-1}s_{n-2}\cdots s_2s_1$.

\begin{proposition}\label{prop:isomschubert}
Let $v\leq c$.
Then $\mathcal{I}_v = \init_{\bf w}(\mathcal{I}_v)$.
\end{proposition}

\begin{proof}
Recall that
$\mathcal{I}_v=(\{p_I\}_{I \not \le v([\#I])}\cup \{R_{J,L}^k\}_{k,J,L})$
We will show that $R^k_{J,L}-\init_{\bf w}(R^k_{J,L})\in (p_I)_{I \not \le v([\#I])}$ for all $k,J,L$. If $R^k_{J,L}=\init_{\bf w}(R^k_{J,L})$ we are done. Otherwise we have
\[
R^k_{J,L}-\init_{\bf w}(R^k_{J,L})=\sum_{
\begin{smallmatrix}
(J',L')\in \mathcal{L}_{J,L}^k\\ \{l_{r_1},\dots,l_{r_k}\}\cap [e,d-1]\not =\emptyset\\
\end{smallmatrix}
}p_{J'}p_{L'}\not =0.
\]
We claim that in this case $L'\not \le v([d])$ holds. Note that $\{l_{r_1},\dots,l_{r_k}\}\cap [e,d-1]\not =\emptyset$ implies in particular that there exists $x\in [e,d-1]$ with $x\not \in L'=(L\setminus(l_{r_1},\dots,l_{r_k}))\cup(j_1,\dots,j_k)$. By \eqref{eq:coxeter},
\begin{eqnarray*}
v\le c \Leftrightarrow v([d])=[d-1]\cup \{b\} \text{ with } d\le b \le n
\end{eqnarray*}
it follows that $p_{L'}\in (p_I)_{I \not \le v([\#I])}$. And further, $R^k_{J,L}-\init_{\bf w}(R^k_{J,L})\in (p_I)_{I \not \le v([\#I])}$.
Hence, 
$$
\mathcal I_v = (R_{J,L}^k)_{k,J,L}+(p_I)_{I\not \le v([\#I])} = (\init_{\bf w}(R^k_{J,L}))_{k,J,L}+(p_I)_{I\not \le v([\#I])} \subseteq  \init_{\bf w}(\mathcal I_v).
$$
Consider any term order $<$ so that $\init_<(\mathcal I_v)=\init_<(\init_{\bf w}(\mathcal I_v))$. Such a term order exists as $\mathcal I_v$ is homogeneous.
Then the reduced Gr\"obner basis for $\mathcal I_v$ with respect to $<$ is also a reduced Gr\"obner basis for $\init_{\bf w}(\mathcal I_v)$ with respect to $<$. 
As reduced Gr\"obner bases are unique the claim follows.
\end{proof}

\section{Criteria for reducibility}\label{sec:reducibility}
In this section we examine when  Schubert varieties become reducible after degenerating. We give a number of sufficient conditions for certain monomials of degree two to be contained in the initial ideal $\init_{\bf w}(\mathcal{I}_w)$ for $w\in S_n$ making repeated use of \eqref{eq:containment}.

\begin{definition}
Let $w\in S_n$. A monomial $f=\prod_{J\subset [n]} \epsilon_J p_J\in \init_{\bf w}(\mathcal I_w)$, where $\epsilon_J\in\{0,1\}$, is called an \emph{honest monomial} if $f$ has degree at least 2 and $f\not\in (p_I)_{I\not \le w([\#I])}$.
\end{definition}

The following Lemma is straightforward:
\begin{lemma}\label{lem:honestmono}
Let $w\in S_n$. If $\init_{\bf w}(\mathcal I_w)$ contains an honest monomial then it is reducible.
\end{lemma}

\subsection{Relations between $\Gr(1,\mathbb C^n)$ and $\Gr(2,\mathbb C^n)$}
We start the discussion by focusing on very special Pl\"ucker relations, namely those between Pl\"ucker coordinates on $\Gr(1,\mathbb C^n)$ and on $\Gr(2,\mathbb C^n)$. In this case, we can classify the $w\in S_n$ for which $\init_{\bf w}(\mathcal{I}_w)$ contains an honest monomial of this form.
\medskip

For $v\in S_n$ denote by $\overline v$ the minimal length representative of the coset of $v$ in $S_n/\langle s_2,s_3.\dots s_{n-1}\rangle$ and $\overline{\overline v}$ the minimal length representative of the coset of $v$ in $S_n/\langle s_1,s_3,s_4,\dots, s_{n-1}\rangle$.

\begin{theorem}\label{thm:gr12}
Let $v\in S_n$ and $1<j<k\le n$. Then $\init_{\bf w}(\mathcal{I}_v)$ contains the honest monomial $p_{\{j\}}p_{\{1,k\}}$ 
if and only if $v$ satisfies
\begin{eqnarray*}
    s_{j-1}s_{j-2}\cdots s_2s_1\le\overline{v}\le s_{k-2}s_{k-3}\cdots s_2s_1 \ \text{ and } \
    s_{k-1}s_{k-2}\cdots s_3s_2\le\overline{\overline v}.
\end{eqnarray*}
\end{theorem}

The conditions on $\overline{v}$ and $\overline{\overline{v}}$ in Theorem~\ref{thm:gr12} are depicted for $S_4$ with $j=2,k=4$ in Figure~\ref{fig:thmgr12}.

\begin{center}
\begin{figure}[h]
\begin{tikzpicture}[scale=.8]
\node at (0,0){\tiny$1$};
\node[magenta] at (0,1){\tiny$2=s_1(1)$};
\node[magenta] at (0,2){\tiny$3=s_2s_1(1)$};
\node at (0,3){\tiny$4$};

\draw (0,.25)-- (0,.75);
\draw[magenta] (0,1.25) -- (0,1.75);
\draw (0,2.25) -- (0,2.75);

\begin{scope}[xshift= 5cm]

\node at (0,0){\tiny$\{1,2\}$};
\node at (0,1){\tiny$\{1,3\}$};
\node[blue] at (-1.75,2){\tiny$s_3s_2(\{1,2\})=\{1,4\}$};
\node at (1,2){\tiny$\{2,3\}$};
\node[blue] at (0,3){\tiny$\{2,4\}$};
\node[blue] at (0,4){\tiny$\{3,4\}$};

\draw (0,.25) -- (0,.75);
\draw (.25,1.25) -- (.75,1.75);
\draw (-.25,1.25) -- (-.75,1.75);
\draw[blue] (-.75,2.25) -- (-.25,2.75);
\draw (.75,2.25) -- (.25,2.75);
\draw[blue] (0,3.25) -- (0,3.75);

\end{scope}
\end{tikzpicture}
    \caption{The Bruhat posets of $\Gr(1,\mathbb C^4)$ and $\Gr(2,\mathbb C^4)$ with intervals given by \textcolor{magenta}{$s_1\le \overline{v}\le s_2s_1$} and \textcolor{blue}{$s_3s_2\le \overline{\overline{v}}$} as in Theorem~\ref{thm:gr12} for $j=2,k=4$.}
    \label{fig:thmgr12}
\end{figure}    
\end{center}

\vspace{-1cm}
\begin{proof}
To simplify notation, for $a\in[n]$ we denote $p_a:=p_{(a)}$, and for $a,b\in[n]$ we write $p_{a,b}$ instead of $p_{(a,b)}$. We will only consider Pl\"ucker coordinates corresponding to increasing sequences in this proof and hence adapt the signs.

Consider for $1\le i<j<k\le n$ the Pl\"ucker relation
$R^1_{(i),(j,k)}=p_ip_{j,k}-p_jp_{i,k}+p_{k}p_{i,j}$. Note that if $\init_{\bf w}(R^1_{(i),(j,k)})=R^1_{(i),(j,k)}$ the relation will not produce an honest monomial in $\init_{\bf w}(\mathcal{I}_w)$ for any $w\in S_n$ as $\mathcal{I}_w$ {is prime}. 
Note that $R^1_{(i),(j,k)}\not =\init_{\bf w}(R^1_{(i),(j,k)})$ only if $i=1$. 
In this case
\[
\init_{\bf w}(p_1p_{j,k}-p_jp_{1,k}+p_{k}p_{1,j})=-p_jp_{1,k}+p_{k}p_{1,j}.
\]
As $j<k$, if $p_j$ vanishes on the Schubert variety $X_v$, then so does $p_k$. Hence, both monomials are zero on $X_v$. Similarly, if $p_{1,j}$ vanishes on $X_v$, then so does $p_{1,k}$. 
Our aim is to determine $v\in S_n$ such that one of the two terms of $\init_{\bf w}(R^1_{(i), (j,k)})$ lies in $(p_I)_{I\not\leq v([\#I])}$ but the other does not as in this case, the ideal $\init_{\bf w}(\mathcal{I}_v)$ contains an honest monomial.
A priori, there are two cases for the restriction of $p_k$ and $p_{1,k}$ to $X_v$:
\begin{enumerate}
    \item $p_{1,k}\not =0$ and $p_k=0$,
    \item $p_{1,k}=0$ and $p_{k} \not=0$.
\end{enumerate}
We will show that in fact the second case can never happen. Both cases yield conditions on $\overline{v}$ and $\overline{\overline v}$ (keeping also in mind that we do not want $p_j$ and $p_{1,j}$ to vanish). In the first case we have the following conditions
\begin{eqnarray}\label{eq:critv1}
    s_{j-1}s_{j-2}\cdots s_2s_1\le\overline{v}\le s_{k-2}s_{k-3}\cdots s_2s_1 \ \text{ and } \
    s_{k-1}s_{k-2}\cdots s_3s_2\le\overline{\overline v},
\end{eqnarray}
respectively, in the second case we have
\begin{eqnarray}\label{eq:critv2}
    s_{k-1}s_{k-2}\cdots s_2s_1\le\overline{v}\ \text{ and }  \
    s_{j-1}s_{j-2}\cdots s_3s_2\le\overline{\overline v}\le s_{k-2}s_{k-3}\cdots s_3s_2.
\end{eqnarray}    
Assume $v\in S_n$ is chosen such that the minimal length representatives of the cosets fulfill the inequalities in \eqref{eq:critv2}. Then
\[
s_{k-1}s_{k-2}\cdots s_2 s_1\le v\le s_{k-2}\cdots s_2 x
\]
for some $x\in \langle s_1,s_3,\dots,s_{n-1}\rangle$. Observe that $s_{k-1}\cdots s_1(1)=k$ and 
\begin{equation*}
s_{k-2}\cdots s_2 x(1) =\left\{
\begin{array}{ll}
1 &\text{if }s_1x>x\\
k-1 & \text{if } s_1x<x.
\end{array}\right.
\end{equation*}
With the notation as in \eqref{eq:w[i,j]} this implies 
$(s_{k-1}\cdots s_1)^{1,k}=1> (s_{k-2}\cdots s_2 x)^{1,k}=0$.
But $s_{k-1}\cdots s_1 \le s_{k-2}\cdots s_2 x$, contradicting \eqref{eq:BBequiv}. Hence, case \eqref{eq:critv2} never applies.
\end{proof}

\begin{remark}\label{rem:sl3_case}
Theorem~\ref{thm:gr12} is enough to detect all Schubert varieties in $\Flag_3\hookrightarrow \Gr(1,\mathbb C^3)\times \Gr(2,\mathbb C^3)$ which become reducible under Feigin's degeneration. In fact, the only Schubert variety having this property is the one indexed by $s_1s_2$. All the other permutations but the longest element (which indexes the Schubert variety corresponding to the irreducible variety $\Flag_n^a$) are $\leq c=s_2s_1$ and hence, by Proposition \ref{prop:isomschubert}, are irreducible.
\end{remark}

\subsection{Monomials from other relations}
Theorem~\ref{thm:alllemma}~(1) to (5) provide sufficient conditions on $w\in S_n$ for the initial ideal $\init_{\bf w}(\mathcal{I}_w)$ to contain a degree two honest monomial originating from a Pl\"ucker relation between Pl\"ucker coordinates on adjacent Grassmannians, that is $\Gr(k,\mathbb C^n)$ and on $\Gr(k+1,\mathbb C^n)$ for suitable $k$. Notice that here we are only producing sufficient conditions, so that for $k=1$ we clearly obtain a weaker result than Theorem \ref{thm:gr12}.  
Theorem~\ref{thm:alllemma}~(6) and (7) deal with Pl\"ucker relations between not necessarily adjacent Grassmannians.

Table~\ref{tab:S4} (resp. Table~\ref{tab:S5} in the appendix) show to which permutations $w\in S_4$ (resp. $S_5$) each one of the points of Theorem \ref{thm:alllemma} applies. The computations for these were performed in \emph{Sage}\cite{sagemath} and \emph{Macaulay2}\cite{M2}.

Let $w\in S_n$. In the following, it will be convenient to set  $w([0]):=\emptyset$. Moreover, since $\init_{\textbf w}(\mathcal{I}_{e})=\mathcal{I}_e$, we can exclude the case $w=e$ right away in the following theorem.

\begin{theorem}\label{thm:alllemma}
Let $w\in S_n\setminus\{e\}$. If one of the following conditions holds for $w$, then $\init_{\bf w}(\mathcal{I}_w)$ contains an honest monomial of degree 2:
\begin{enumerate}
\item \label{lem:monomial1} 
    there exist $i\in [n-1]$ with $ws_i>w$ and $j\in [n]$ such that 
    \[
    i,j\le w(i), i\not=j\text{ and }i,j\not\in w([i-1])\cup \{w(i+1)\};
    \]
\item \label{lem:monomial2} 
     there exist $i\in[3,n-1]$ with $ws_i>w$ and $l,x\in[n]$ with $x\not=i-1, l\le w(i)$ and $w(i+1)\le x,i-1$, such that
    \[
    i-1,x\in w([i-1])\cup\{w(i+1)\}\text{ and }l\not\in w([i-1])\cup\{w(i+1)\};
    \]
\item \label{lem:monomial3} 
     there exist $j\in[2,n-1]$ with $s_jw>w$ and $i\in[n-1],i<j$ such that
    \begin{eqnarray*}
     j\in w([i]), i\not \in w([i]),
    \text{ and } j+1\le w(i+1);
    \end{eqnarray*}
\item \label{lem:monomial4}
     there exists $i\in[n-2]$ with $s_iw<w$ and $j\in [n]$ such that \[
    i,j\not \in w([i+1]), j\le w(i+2), i+1 \in w([i+1]) \text{ and } i+1<j;
    \]    
\item \label{lem:monomial5}
     there exist $i\in[2,n-1]$ and $l\in[2,n], l>i$ with
    \[
    i\not\in w([i+1]),l\in w([i]),l>w(i+1) \text{ and  } i>w(i+1);
    \]    
\item \label{lem:easylemma}
 for $i\in[n]$, minimal with $w(i)\not=i$, it holds $w(i)<n$ and, for the minimal $j\in[i+1,n-1]$ such that $w(j)>w(i)$, it holds
$
 w(i)\not \in [j-1];
 $
\item \label{lem:w(i)=n}
      for $i\in[n]$, minimal with $w(i)\not=i$, it holds $w(i)=n$ and, for the minimal $j\in[i+2, n-1]$, such that  $w(j)>w(i+1)$, it holds $i\not\in w([i+1,j-1])$.
\end{enumerate}

\end{theorem}

\begin{proof}

\noindent
\begin{enumerate}
\item  
    Assume there exist $i,j$ fulfilling the conditions above. Let $J$ be any sequence such that  $F(J)=w([i-1])\cup\{j\}$ and $j_1=j$, and let $L$ be any sequence such that $F(L)=w([i-1])\cup\{i,w(i+1)\}$. Then the Pl\"ucker relation $R^1_{J,L}$ equals
    \begin{eqnarray*}
    p_{J}p_{L}-p_{(J\setminus (j))\cup(i)}p_{(L\setminus(i))\cup(j)} -p_{(J\setminus(j))\cup (w(i+1))} p_{(L\setminus(w(i+1)))\cup(j)}.
    \end{eqnarray*}
    Taking the initial form with respect to ${\bf w}$ we obtain
    \begin{eqnarray*}
    \init_{\bf w}(R^1_{J,L})=p_{J}p_{L} -p_{(J\setminus(j))\cup (w(i+1))} p_{(L\setminus(w(i+1)))\cup(j)}.
    \end{eqnarray*}
    Restricting to $X_w$, we have $p_{(J\setminus(j))\cup (w(i+1))}=p_{(w([i-1]),w(i+1))}=0$ as $ws_i>w$ and so $\init_{\bf w}(\mathcal I_w)$ contains the monomial $p_Jp_L$.
\item
    Assume such $i,l,x$ exist. Let $J$ be any sequence such that $F(J)=(w([i-1])\cup\{ w(i+1)\})\setminus\{i-1\}$ and $j_1=x$, and let $L$ be any sequence such that $F(L)=(w([i-1])\cup\{w(i+1),l\})\setminus\{x\}$ the Pl\"ucker relation $R^1_{J,L}$, i.e.
    \begin{eqnarray*}
    p_{J}p_{L} -p_{(J\setminus(x))\cup(i-1)}
    p_{(L\setminus(i-1))\cup(x)}- p_{(J\setminus(x))\cup(l)}
    p_{(L\setminus(l))\cup(x)}.
    \end{eqnarray*}
    Taking the initial form with respect to ${\bf w}$ we obtain
    \begin{eqnarray*}
    \init_{\bf w}(R^1_{J,L})= p_{J}p_{L} - p_{(J\setminus(x))\cup(l)}
    p_{(L\setminus(l))\cup(x)}.
    \end{eqnarray*}
    Note that $(F(L)\setminus\{l\})\cup\{x\}=w([i-1])\cup\{w(i+1)\}$ and so restricting to $X_w$ we have $p_{(L\setminus(l))\cup(x)}=0$ as $ws_i>w$. So $\init_{\bf w}(\mathcal I_w)$ contains the monomial $p_Jp_L$.
\item
    Assume such $i$ and $j$ exist and take $J$ any sequence such that $F(J)=w([i])$ and $j_1=j$, and  $L$ any sequence such that $F(L)=(w([i])\cup\{i,j+1\})\setminus\{j\}$. Note that $j\in w([i])$ and $s_jw>w$ imply $ j+1\not\in w([i+1])$. Then
    \begin{eqnarray*}
    R^1_{J,L}= p_{J}p_{L} -p_{(J\setminus(j))\cup(i)}p_{(L\setminus(i))\cup(j)} -p_{(J\setminus(j))\cup(j+)}p_{(L\setminus(j+1))\cup(j)}.
    \end{eqnarray*}
    Taking the initial form with respect to ${\bf w}$ we obtain
    \begin{eqnarray*}
    \init_{\bf w}(R^1_{J,L})= p_{J}p_{L}- p_{(J\setminus(j))\cup(j+1)}p_{(L\setminus(j+1))\cup(j)}.
    \end{eqnarray*}
    As $(J\setminus(j))\cup(j+1)\not\leq w([\#J])$ restricting to $X_w$ we have $p_{(w([i])\setminus(j))\cup(j+1)}=0$. Hence, $\init_{\bf w}(\mathcal I_w)$ contains the monomial $p_Jp_L$.
\item   
    Assume such $i$ and $j$ exist and consider $L$ any sequence such that $F(L)=w([i+1])\cup\{j\}$, and $J$ any sequence such that $F(J)=s_iw([i+1])=(w([i+1])\setminus\{i+1\})\cup\{i\}$ and $j_1=i$. Then 
    \begin{eqnarray*}
    R^1_{J,L}= p_{J}p_{L} -p_{(J\setminus(i))\cup(i+1)} p_{(L\setminus(i+1))\cup(i)} - p_{(J\setminus(i))\cup(j)}p_{(L\setminus(j))\cup(i)}
    \end{eqnarray*}
    Taking the initial form with respect to ${\bf w}$ yields
    \begin{eqnarray*}
    \init_{\bf w}(R^1_{J,L})= p_{J}p_{L} - p_{(J\setminus(i))\cup(j)}p_{(L\setminus(j))\cup(i)}
    \end{eqnarray*}
    Now $(J\setminus(i))\cup(j)=(w([i+1])\setminus(i+1))\cup(j)$, but restricting to $X_w$ we have $p_{(J\setminus(i))\cup(j)}=0$ as $j>i+1$. Hence, $\init_{\bf w}(\mathcal I_w)$ contains the monomial $p_Jp_L$. 
\item
    Assume such $i,l$ exist, take $J=w([i])$ and $L=(w([i+1])\setminus\{l\})\cup\{i\}$. Consider the relation $R^1_{J,L}$:
    \begin{eqnarray*}
     p_{J}p_{L} -p_{(J\setminus(l))\cup(i)}p_{(L\setminus(i))\cup(l)} - p_{(J\setminus(l))\cup(w(i+1))}p_{(L\setminus(w(i+1)))\cup(l)}.
    \end{eqnarray*}
    Taking the initial form with respect to ${\bf w}$ yields
    \begin{eqnarray*}
    \init_{\bf w}(R^1_{J,L}) = p_{J}p_{L} - p_{(J\setminus(l))\cup(w(i+1))}p_{(L\setminus(w(i+1)))\cup(l)}.
    \end{eqnarray*}
    Restricting to $X_w$ we have $(F(L)\setminus\{w(i+1)\})\cup\{l\}=(F(w([i+1])\setminus\{w(i+1)\})\cup\{i\}$ and $p_{(w([i+1])\setminus(w(i+1)))\cup(i)}=0$ as $i>w(i+1)$. So $\init_{\bf w}(\mathcal I_w)$ contains the monomial $p_Jp_L$.
\item    
    First note that {as $w\not= e$ we have that} $w(i)\not =i$ in particular implies $i<n$. Consider $J$ any sequence such that $F(J)=w([i])=[i-1]\cup\{w(i)\}$ with $j_1=w(i)$. Let $L$ be any sequence such that $F(L)=[j-1]\cup \{w(j)\}$. As $w(i)\not \in [j-1]$ implies $w(i)>j-1$ and so $w(j)>w(i)>j-1$, then the set $[j-1]\cup \{w(j)\}$ has cardinality $j$.
    So, 
    \begin{eqnarray*} \quad \quad
    R^1_{J,L}= p_{J}p_{L}-p_{(w(j),[i-1])} p_{(L\setminus(w(j)))\cup(w(i))} -\sum_{r\in[i,j-1]}p_{(r,[i-1])} p_{(L\setminus(r))\cup(w(i))}.
    \end{eqnarray*}
    Taking the initial form with respect to ${\bf w}$ yields
    \begin{eqnarray*}
    \init_{\bf w}(R^1_{J,L})=p_{J}p_{L}-p_{(w(j),[i-1])} p_{(L\setminus(w(j)))\cup(w(i))}.
    \end{eqnarray*}
    Since $w(j)>w(i)$, the coordinate $p_{(w(j),[i-1])}$ vanishes in the coordinate ring of $X_w$, so that $\init_{\bf w}(R^1_{J,L})\in\init_{\bf w}(\mathcal{I}_w)$ is a monomial.
\item   
    Consider $J$ any sequence such that $F(J)=[i]\cup\{n\}=w([i])\cup \{i\}$ such that $j_1=i$, and let $L$ be any sequence such that $F(L)=[i-1]\cup[i+1,j-1]\cup\{w(j),n\}$. Note that $L\le w([j])$ as $i\not\in w([i+1,j-1])$, and hence we get
    \[
    R^1_{J,L}= p_{J}p_{L}-p_{(w(j),w([i]))} p_{(L\setminus(w(j)))\cup(i)}
 -\sum_{r\in[i+1,j-1]}p_{(r,w([i]))} p_{(L\setminus(r))\cup(i)}
    \]
    with initial term $\init_{\bf w}(R^1_{J,L})=p_{J}p_{L}-p_{(w(j),w([i]))} p_{(L\setminus(w(j)))\cup(i)}$. 
    Further observe that $w(j)>w(i+1)\ge i$, which implies that $p_{(w(j),w([i]))}$ vanishes in the coordinate ring of $X_w$. Then $R^1_{J,L}$ produces a monomial.
\end{enumerate}
\end{proof}

\begin{remark}
In principle, we could have assumed $i\in\{2, 3, \ldots, n-1\}$ in Theorem~\ref{thm:alllemma}~(2). Instead, we exclude the case $i=2$, since it never happens under the other assumptions, for which we would have $w(3)\le 1$ and $ws_2(2)=w(3)>w(2)$ contradicting each other.
\end{remark}

\begin{remark}In the points (6) and (7) of  Theorem~\ref{thm:alllemma}, such a $j$ need not exists, in which case the  criterion would
simply not apply.
\end{remark}

\subsubsection{Efficiency of the various criteria from Theorem~\ref{thm:alllemma}.} 
We want to comment here on how efficient the various criteria of Theorem~\ref{thm:alllemma} are, based on the data we have collected for $S_4$ (see Table~\ref{tab:S4}) and $S_5$ (see Table~\ref{tab:S5}). The data can be found at the homepage:
\url{https://www.matem.unam.mx/~lara/schubert/}.
 
For $n=4$, there are 11 permutations $w$ such that at least one Pl\"ucker relation degenerates to a monomial. In the $S_5$-case, this happens for 85 permutations. 

 Among the criteria collected in Theorem~\ref{thm:alllemma},  point (6) seems to be the most powerful: it detects 9 out of  11 permutations  for $S_4$, and 65 out of 85 for $S_5$. To cover the missing two permutations for $S_4$ it is enough to combine  Theorem~\ref{thm:alllemma}~(6) with one of the points (1),(4),(7) and one between (2) and (5). So that it is enough to apply three of our criteria to find all $w\in S_4$ such that $\init_{\textbf w}(\mathcal{I}_w)$ contains a Pl\"ucker relation which degenerates to a monomial.  
 
 Theorem~\ref{thm:alllemma} (1) picks 9 out of  11 permutations in $S_4$, and 64 out of 85 for $S_5$.

Theorem~\ref{thm:alllemma}~(3) covers 8 out of 11 permutations yielding monomial initial ideals for $S_4$ and 57 out of 85 for $S_5$.

Theorem~\ref{thm:alllemma}~(4) detects 4 permutations for $S_4$ and 36 permutations for $S_5$.

Theorem~\ref{thm:alllemma}~(2) and (5) both finds 2 permutations for $n=4$ and 22 for $n=5$, but the elements they see are different.

Finally, Theorem~\ref{thm:alllemma}~(7) applies to only one permutation, resp. 8 permutations,  in the $n=4$, resp. $n=5$, case, but it is necessary to cover all the permutations in $S_5$ containing monomial degenerate Pl\"ucker relations. For example, it is the only one among our criteria which  can be applied to 
$s_1s_2s_3s_4s_3s_1s_2s_1$.

\begin{table}
\begin{tabular}{ll | c | lllllll}
$w$ one-line & $w$ red. word & mono & (1) & (2) & (3) & (4) & (5) & (6)& (7) \\ \hline
$[1, 2, 3, 4]$ & $1$ & $-$ & $-$ & $-$ & $-$ & $-$ & $-$ & $-$ & $-$ \\
$[1, 2, 4, 3]$ & $s_{3}$ & $-$ & $-$ & $-$ & $-$ & $-$ & $-$ & $-$ & $-$ \\
$[1, 3, 2, 4]$ & $s_{2}$ & $-$ & $-$ & $-$ & $-$ & $-$ & $-$ & $-$ & $-$ \\
$[1, 3, 4, 2]$ & $s_{2}s_{3}$ & $\times$ & $\times$ & $-$ & $\times$ & $-$ & $-$ & $\times$ & $-$ \\
$[1, 4, 2, 3]$ & $s_{3}s_{2}$ & $-$ & $-$ & $-$ & $-$ & $-$ & $-$ & $-$ & $-$ \\
$[1, 4, 3, 2]$ & $s_{2}s_{3}s_{2}$ & $-$ & $-$ & $-$ & $-$ & $-$ & $-$ & $-$ & $-$ \\
$[2, 1, 3, 4]$ & $s_{1}$ & $-$ & $-$ & $-$ & $-$ & $-$ & $-$ & $-$ & $-$ \\
$[2, 1, 4, 3]$ & $s_{3}s_{1}$ & $-$ & $-$ & $-$ & $-$ & $-$ & $-$ & $-$ & $-$ \\
$[2, 3, 1, 4]$ & $s_{1}s_{2}$ & $\times$ & $\times$ & $-$ & $\times$ & $-$ & $-$ & $\times$ & $-$ \\
$[2, 3, 4, 1]$ & $s_{1}s_{2}s_{3}$ & $\times$ & $\times$ & $-$ & $\times$ & $\times$ & $-$ & $\times$ & $-$ \\
$[2, 4, 1, 3]$ & $s_{3}s_{1}s_{2}$ & $\times$ & $\times$ & $-$ & $\times$ & $-$ & $-$ & $\times$ & $-$ \\
$[2, 4, 3, 1]$ & $s_{1}s_{2}s_{3}s_{2}$ & $\times$ & $\times$ & $-$ & $\times$ & $\times$ & $-$ & $\times$ & $-$ \\
$[3, 1, 2, 4]$ & $s_{2}s_{1}$ & $-$ & $-$ & $-$ & $-$ & $-$ & $-$ & $-$ & $-$ \\
$[3, 1, 4, 2]$ & $s_{2}s_{3}s_{1}$ & $\times$ & $-$ & $-$ & $\times$ & $-$ & $-$ & $\times$ & $-$ \\
$[3, 2, 1, 4]$ & $s_{1}s_{2}s_{1}$ & $-$ & $-$ & $-$ & $-$ & $-$ & $-$ & $-$ & $-$ \\
$[3, 2, 4, 1]$ & $s_{1}s_{2}s_{3}s_{1}$ & $\times$ & $\times$ & $-$ & $-$ & $\times$ & $-$ & $\times$ & $-$ \\
$[3, 4, 1, 2]$ & $s_{2}s_{3}s_{1}s_{2}$ & $\times$ & $\times$ & $\times$ & $\times$ & $-$ & $\times$ & $\times$ & $-$ \\
$[3, 4, 2, 1]$ & $s_{1}s_{2}s_{3}s_{1}s_{2}$ & $\times$ & $\times$ & $-$ & $\times$ & $-$ & $-$ & $\times$ & $-$ \\
$[4, 1, 2, 3]$ & $s_{3}s_{2}s_{1}$ & $-$ & $-$ & $-$ & $-$ & $-$ & $-$ & $-$ & $-$ \\
$[4, 1, 3, 2]$ & $s_{2}s_{3}s_{2}s_{1}$ & $-$ & $-$ & $-$ & $-$ & $-$ & $-$ & $-$ & $-$ \\
$[4, 2, 1, 3]$ & $s_{3}s_{1}s_{2}s_{1}$ & $-$ & $-$ & $-$ & $-$ & $-$ & $-$ & $-$ & $-$ \\
$[4, 2, 3, 1]$ & $s_{1}s_{2}s_{3}s_{2}s_{1}$ & $\times$ & $\times$ & $-$ & $-$ & $\times$ & $-$ & $-$ & $\times$ \\
$[4, 3, 1, 2]$ & $s_{2}s_{3}s_{1}s_{2}s_{1}$ & $\times$ & $-$ & $\times$ & $-$ & $-$ & $\times$ & $-$ & $-$ \\
$[4, 3, 2, 1]$ & $s_{1}s_{2}s_{3}s_{1}s_{2}s_{1}$ & $-$ & $-$ & $-$ & $-$ & $-$ & $-$ & $-$ & $-$ \\
\hline
24 & & 11 & 9 & 2 & 8 & 4 & 2 & 9 & 1\\
\end{tabular}
\caption{Applying Theorem~\ref{thm:alllemma} to $S_4$}\label{tab:S4}
\end{table}

\subsection{Pl\"ucker relations not degenerating to monomials}In this section we study some cases in which none of the Pl\"ucker relations produces a monomial in the defining ideal $\init_{\bf w}(\mathcal{I}_w)$. Clearly, this  does not have to be equivalent to the irreducibility of the degeneration, but it turns out to be the case for $n=3$ (by Remark \ref{rem:sl3_case}) and $n\in\{4,5\}$ (by Macaulay2 computations). We do not know whether such an equivalence holds in general.

We have seen in \S\ref{sec:deg-vs-orig}, that if $v\leq c=s_{n-1}s_{n-2}\cdots s_2s_1$, then the initial ideal $\init_{\bf w}(\mathcal{I}_w)$  coincides with $\mathcal{I}_w$.  In the following proposition we will show  that if we multiply $c$ on the right by simple reflections $s_{k_1}, \ldots, s_{k_r}$ which commute pairwise and each appear at most once, then none of the Pl\"ucker relations  degenerates to a monomial in $\init_{\bf w}(\mathcal{I}_{cs_{k_1}\cdots s_{k_r}})$.
\smallskip

Table~\ref{tab:S4} (resp. Table~\ref{tab:S5} in the appendix) show which statements apply to which elements of $ S_4$ (resp. $S_5$). 

\begin{proposition}\label{prop:cox+refl1}
For any $h\in[n-1]$, none of the Pl\"ucker relations  degenerates to a monomial in $\init_{\bf w}(\mathcal{I}_{cs_{h}})$.
\end{proposition}
\begin{proof}
First of all notice that if $h=1$, then $cs_1<c$ and the claim follows from Proposition \ref{prop:isomschubert}, which says that $\init_{\bf w}(\mathcal{I}_c)=\mathcal{I}_c$.

If $h\in[2,n-1]$, then $cs_{h}>c$. In this case, if $J\le c([\#J])$ and $L\le c([\#L])$, then $\init_{\bf w}(R^m_{J,L})$ being a monomial on $X^a_{cs_{h}}$ implies that it is a monomial on $X_{c}^a$ too. But this is not possible, again by Proposition \ref{prop:isomschubert}.  Therefore we can assume that $L\nleq c([\#L])$ or  $J\nleq c([\#J])$.  We set  $k:=h-1\in[n-2]$ for convenience.

 Recall that for any $i\in[k]\cup[k+2,n-1]$
\begin{eqnarray*}
cs_{k+1} /\langle s_1,\dots,s_{i-1},s_{i+1},\dots,s_{n-1}\rangle &=&s_r\cdots s_{i} /\langle s_1,\dots,s_{i-1},s_{i+1},\dots,s_{n-1}\rangle \\
&=&c/\langle s_1,\dots,s_{i-1},s_{i+1},\dots,s_{n-1}\rangle.
\end{eqnarray*} 
In one-line notation $cs_{k+1}=[n,1,\dots,k-1,k+1,k,k+2,\dots,n-1]$. Hence, if $I\le cs_{k+1}([\#I])$, but $I\nleq c([\#I])$, then $\#I=k+1$ and it must hold 
\begin{equation}\label{eq:varnewJ}
F(I)=[k-1]\cup \{k+1,i\} \hbox{ with }i\in[k+2,n]. 
\end{equation}
Therefore a Pl\"ucker $R^m_{J,L}$ can  produce a monomial in $\init_{\bf w}(\mathcal{I}_{cs_h})$ only if $J$ is a sequence such that $F(J)=[k-1]\cup \{k+1,j\}$ with $j_1=j$ or $F(L)=[k-1]\cup \{k+1,l\}$ for $j,l\in[k+2,n]$. 
If $\#J=\#L$, then $\init_{\bf w}(R^m_{J,L})=R^m_{J,L}$, hence we only have to consider the case $\#J<\#L$.

Let $\#L=p>k+1$, then by \eqref{eq:varnewJ} we have
$F(J)=[k-1]\cup\{k+1,j\}$ and $F(L)=[p-1]\cup\{l\}$
for $j_1=j\in[k+2,n]$ and $l\in[p,n]$. Note that $j\in J$ is the only possible element to swap for elements in $L$ non-trivially, so that we impose $j\not \in L$ (otherwise $R^m_{J,L}=0$ for any $m$). Remember that we may assume $j\in[p,n]$. Then
\begin{eqnarray}\label{eq:newJ}
\init_{\bf w}(R^1_{J,L})=p_Jp_L-p_{(J\setminus(j))\cup(l)}p_{(L\setminus(l))\cup(j)}-p_{(J\setminus(j))\cup(k)}p_{(L\setminus(k))\cup(j)}.
\end{eqnarray}
As $[k-1]\cup\{k+1,l\}\le cs_{k+1}([k+1])$ and $[k-1,p-1]\cup\{j\}\le cs_{k+1}([p])$ at least two terms are non-zero on $X_{cs_{k+1}}$. 

Now, assume $\#L=k+1$ and $\#J=q<k+1$. Then we have 
\[
F(L)=[k-1]\cup\{k+1,l\} \text{ and } F(J)=[q-1]\cup\{j\},
\]
for $j=j_1,l\in[k+2,n]$ and $j\not\in L$ in order for the relation to be non-trivial. We obtain
\begin{eqnarray}\label{eq:newL}
\init_{\bf w}(R^1_{J,L})=p_Jp_L-p_{(J\setminus(j))\cup(k+1)}p_{(L\setminus(k+1))\cup(j)}-
p_{(J\setminus(j))\cup(l)}p_{(L\setminus(l))\cup(j)}.
\end{eqnarray}
As $[q-1]\cup \{l\}\le cs_{k+1}([q])$ and $[k-1]\cup\{k+1,j\}\le cs_{k+1}([k+1])$, the relation $R^1_{J,L}$ does not degenerate to a monomial.
\end{proof}

\begin{corollary}\label{cor:purediff}
Let $h\in[n-1]$. Then $\init_{\bf w}(\mathcal{I}_{cs_{h}})$ is a pure difference ideal in the quotient $\mathbb{C}[p_{I}]/(p_I\mid I\nleq cs_{h}([\#I]))$.
\end{corollary}
\begin{proof}
First note that if $h=1$, then by Proposition \ref{prop:isomschubert} $\init_{\bf w}(\mathcal{I}_{cs_{1}})=\mathcal{I}_{cs_1}$.
The Pl\"ucker relations involving non-vanishing Pl\"ucker coordinates on $X_{cs_1}$ are for $q<p\le j<l\le n$ the following pure differences
\[
p_{[q-1]\cup\{j\}}p_{[p-1]\cup\{l\}}-p_{[q-1]\cup\{l\}}p_{[p-1]\cup\{j\}}.
\]
Notice that the index sets of the Pl\"ucker coordinates in the above equation (as well as in the rest of this proof) are sets, and hence by convention, as sequences they are arranged in an increasing order, while in the proof of the previous result we always had $j=j_1$. This only affect the relation by a global sign.

If $h\in[2,n-1]$, we can set again $k:=h-1$. In the proof of Proposition~\ref{prop:cox+refl1} we have seen in equations \eqref{eq:newJ} and \eqref{eq:newL} the form of the additional relations for $cs_{k+1}$. Note that in \eqref{eq:newJ} we have $[k-1]\cup[k+1,p-1]\cup\{j,l\}\not \le cs_{k+1}([p])$ and hence, the middle term vanishes on $X_{cs_{k+1}}$. Similarly observe for \eqref{eq:newL} that $[k-1]\cup\{j,l\}\not\le cs_{k+1}([k+1])$ as $j,l\ge k+2$. So all generators of $\init_{\bf w}(\mathcal{I}_{cs_{k+1}})$ are pure differences in $\mathbb{C}[p_{I}\vert 
]/(p_{I}\mid I\not\leq cs_{k+1}(\#I))$.
\end{proof}

\begin{remark}
Note that while $\init_{\bf w}(\mathcal{I}_w)$ and $\mathcal{I}_w$ have the same generators for $w\le c$, this is not true for $cs_{k+1}$ with $k\ge1$. Here taking the initial ideal with respect to ${\bf w}$ modifies the generators.
\end{remark}

The following proposition generalizes Proposition	 \ref{prop:cox+refl1} to a product of pairwise distinct commuting simple reflections.

\begin{proposition}\label{prop:cox+refl}
Take $k_1,\dots,k_r\in [n-1]$ with $|k_i-k_j|>1$ for all $i\neq j$,  then none of the Pl\"ucker relations  degenerates to a monomial in $\init_{\bf w}(\mathcal{I}_{cs_{k_1}\cdots s_{k_r}})$.
\end{proposition}
\begin{proof}
We may assume $k_1<k_2<\ldots<k_r$ without loss of generality.
Moreover, since we are multiplying by pairwise distinct commuting reflections, and as Pl\"ucker relations only involve pairs of Grassmannians, it is enough to consider the cases $r=1,\ 2$. The case $r=1$ was dealt with in Proposition~\ref{prop:cox+refl1}, so we are left with $r=2$.

We consider two cases:  firstly, we deal with the case $k_1=1$, and then we suppose $k_1\neq 1$.

If $k_1=1$, $cs_1<c$ can be identified with the  Coxeter element $\tilde{c}=\tilde{s}_{n-2}\ldots \tilde{s}_{1}$ in $S_{n-1}$ (via $s_i\mapsto \tilde{s}_{i-1}$ for $i\in[2,n-1]$). In this case, $cs_1s_{k_2}\in\langle s_2, \ldots, s_{n-1}\rangle$ and, by Corollary \ref{cor:smallflag}, we have $\init_{\bf w}(\mathcal{I}_{cs_1s_{k_2}})=\init_{\bf w}(\mathcal{I}_{\tilde{c}\tilde{s}_{k_2}})$. We then apply Proposition~\ref{prop:cox+refl1} to obtain the claim. 

Now denote $k_1:=k+1$ and $k_2:=g+1$ and recall, that by assumption $k< g+1$. As in the proof of Proposition \ref{prop:cox+refl1}, we only have to deal with Pl\"ucker relations $R^m_{J,L}$ with $\#J\neq\#L$, where $J\nleq cs_{k+1}s_{g+1}([\#J])$ or $L\nleq cs_{k+1}s_{g+1}([\#L])$. We can further reduce to the case $\#J=k+1$, $j_1=j$, and $\#L=g+1$, otherwise the Pl\"ucker relations are the same as the ones considered in Proposition~\ref{prop:cox+refl1}, and the result has been proven above. 

Consider relations $R^m_{J,L}$ with $\#J=k+1,\#L=g+1$ and $J\le cs_{k+1}s_{g+1}([k+1]),J\not \le c([k+1])$ and $L\le cs_{k+1}s_{g+1}([g+1]),L\not \le c([g+1])$. We have shown in Proposition~\ref{prop:cox+refl1} that in this case it must hold
\[
F(J)=[k-1]\cup\{k+1,j\},\quad F(L)=[g-1]\cup\{g+1,l\}
\]
with $j\in[k+2,n]$ 
 and $l\in[g+2,n]$.  In order for the relation to be non-trivial we may assume $j\not\in L$. Since $k+1\in[g-1]$, the only relation to be considered is
\begin{eqnarray*}
R^1_{J,L}&=&p_Jp_L-p_{(J\setminus(j))\cup(l)}p_{(L\setminus(l))\cup(j)}-p_{(J\setminus(j))\cup(g+1)}p_{(L\setminus(g+1))\cup(j)}\\
&\quad&-\sum_{r\in[k+1,g-1]}p_{(J\setminus(j))\cup(r)}p_{(L\setminus(r))\cup(j)}.
\end{eqnarray*}
It degenerates to
\[\init_{\bf w}(R^1_{J,L})=p_Jp_L-p_{(J\setminus(j))\cup(l)}p_{(L\setminus(l))\cup(j)}-p_{(J\setminus(j))\cup(g+1)}p_{(L\setminus(g+1))\cup(j)}.\]
The monomial $p_{(J\setminus(j))\cup(l)}p_{(L\setminus(l))\cup(j)}$ does not vanish on the coordinate ring of $X_{cs_{k+1}s_{l+1}}$ (and thus of $X_{cs_{k_1}\ldots s_{k_r}}$). 
Hence, $\init_{\bf w}(R^1_{J,L})$  is not monomial and this finishes the proof.
\end{proof}

Lemma~\ref{lem:monofreecoxeter} below shows that the Coxeter word $c=s_{n-1}\cdots s_2s_1$ is in fact special among all Coxeter words regarding the degeneration.

\begin{lemma}\label{lem:monofreecoxeter}
Let $w\in S_n$ have a reduced expresion $\underline{w}=s_{i_r}\cdots s_{i_1}$ with $i_k\not =i_l$ for all $k\not =l$. Then none of the Pl\"ucker relations degenerates to a monomial in $\init_{\bf w}(\mathcal{I}_w)$ if and only if $w\le c$.
\end{lemma}
\begin{proof}
"$\Leftarrow$" by Proposition~\ref{prop:isomschubert}.

\noindent
"$\Rightarrow$"
Assume $w=s_{i_r}\ldots s_{i_1}$ is a product of pairwise distinct simple reflections. First note that $w\not \le c$ implies there exists an $i_k\in\{i_1,\dots,i_r\}$ such that $i_k+1=i_{l}$ for $l<k$. We choose $i=i_k$, such that $k$ is minimal with this property. In particular, if there exists $t$ with $i_t+1=i$ then $t<k$. Since $s_i$ commutes with all reflections $s_{i_m}$ with $m>k$, as in this case $i_m\neq i\pm 1$ by minimality of $k$, we observe
\[
w=s_is_{i_r}\dots s_{i_{k+1}}s_{i_{k-1}}\dots s_{i_1}\in s_i\langle s_1,s_2 \ldots,s_{i-1},s_{i+1}, \ldots s_{n-1} \rangle.
\]
 We deduce that $w([i])=[i-1]\cup\{i+1\}$. Moreover, 
notice $w(i+1)\geq i+2$, since $i+1$ is moved only by $s_i$ and $s_{i+1}$, but we apply $s_{i+1}$ first and by hypothesis there are no other occurrences of $s_{i+1}$. We can now produce the degree two monomial in $\init_{\bf w}(\mathcal{I}_w)$ by choosing as $J$ any sequence such that $F(J)=w([i])$ and $j_1=i+1$,  and as $L$ any sequence with $F(L)=[i]\cup\{i+2\}$, so that
\begin{eqnarray*}
R^1_{J,L}&=&p_Jp_L- p_{(J\setminus(i+1))\cup(i+2)}p_{(L\setminus(i+2))\cup(i+1)} -p_{(J\setminus(i+1))\cup(i)}p_{(L\setminus(i))\cup(i+1)},\end{eqnarray*}
\begin{eqnarray*}
\init_{\bf w}(R^1_{J,L})&=&p_Jp_L-p_{(J\setminus(i+1))\cup(i+2)}p_{(L\setminus(i+2))\cup(i+1)}.
\end{eqnarray*}
As $[i-1]\cup\{i+2\}\not\le w([i])$ the second term vanishes on $X_w$.
\end{proof}

\subsection{More and more monomials}
If we can write a permutation $u\in S_n$ as a product of two  permutations $v$, $w$ belonging to two distinct parabolic subgroups which centralize each other, then we can check how a Pl\"ucker relation degenerates on $\mathcal{I}_u$ by looking at the ideals $\mathcal{I}_v$ and $\mathcal{I}_w$. Lemma~\ref{lem:comm.Schubert} concerns defining ideals for Schubert varieties and allows us to deduce Corollary~\ref{cor:commuting}, which suggests an inductive procedure on $n$ to find Schubert varieties that become reducible under Feigin's degeneration.

\begin{lemma}\label{lem:comm.Schubert}
Let $v,w\in S_n$ assume there exist  two  sets of simple reflections $\mathcal{S}_v=\{s_{i_1}, \ldots, s_{i_r}\}$ and $\mathcal{S}_w=\{s_{j_1}, \ldots, s_{j_s}\}$ such that $|i_h-j_l|>1$ for all $h\in[r]$,  $l\in[s]$ with $v\in\langle \mathcal{S}_v\rangle$ and $w\in\langle \mathcal{S}_w\rangle$. Then for all  sequences $J,L$ with $k\le \#J$ we have
\[
R^k_{J,L}\vert_{X_{vw}}=R^k_{J,L}\vert_{X_{v}} \text{ or } R^k_{J,L}\vert_{X_{vw}}=R^k_{J,L}\vert_{X_{w}}.
\]
\end{lemma}

\begin{corollary}\label{cor:commuting}
Let $v,w\in S_n$ assume there exist  two  sets of simple reflections $\mathcal{S}_v=\{s_{i_1}, \ldots, s_{i_r}\}$ and $\mathcal{S}_w=\{s_{j_1}, \ldots, s_{j_s}\}$ such that $|i_h-j_l|>1$ for all $h\in[r]$,  $l\in[s]$ with $v\in\langle \mathcal{S}_v\rangle$ and $w\in\langle \mathcal{S}_w\rangle$. Then
\begin{enumerate}
    \item None of the $R^k_{J,L}$ degenerates to a monomial nor in $\init_{\bf w}(\mathcal I_w)$ neither in $\init_{\bf w}(\mathcal I_v)$, if and only if none of the $R^k_{J,L}$ degenerates to a monomial in $\init_{\bf w}(\mathcal I_{vw})$.
    \item If $\init_{\bf w}(\mathcal{I}_w)$ or $\init_{\bf w}(\mathcal{I}_v)$ contains a monomial degenerate Pl\"ucker relation, then so does $\init_{\bf w}(\mathcal{I}_{vw})$.
\end{enumerate}
\end{corollary}

\begin{remark}From the previous corollary we see that the bigger $n$ is, the more Schubert varieties become reducible after degenerating them \`a la Feigin, since there are several ways of embedding $S_m$ into $S_n$ for $m<n$ as a parabolic subgroup.
Indeed, the number of permutations  $v\in S_n$ such that at least one Pl\"ucker relation degenerates to a monomial in $\init_{\bf w}(\mathcal{I}_v)$ is 0,1,11,85 for $n=2,3,4,5$, respectively. As a curiosity, we mention here that there is exaclty one sequence in the On-Line Encyclopedia of Integer Sequences \cite[Sequence A129180]{OEIS} whose first four terms are 0, 1, 11, 85, namely the \emph{Total area below all Schroeder paths of semilength n}. 
\end{remark}

\section{Degenerate Schubert and Richardson varieties}\label{sec:rich}

In this section we explore how degenerate Schubert varieties behave under the embedding 
of the degenerate flag variety $\Flag_n^a$ into a larger partial flag variety given by Cerulli Irelli and the second author in \cite{C-IL15}. 

\subsection{Degenerate flag varieties and flag varieties of higher rank}
We start by introducing some notation and recalling the main result of \cite{C-IL15}.

Let $\omega_i$ denote the $i$-th fundamental weight for $SL_{2n-2}$ and consider the parabolic subgroup $P:=P_{\omega_1+\omega_3+\dots +\omega_{2n-3}}$ of $SL_{2n-2}$. 
Then, $SL_{2n-2}/P$ is the variety of (partial) flags in $\mathbb{C}^{2n-2}$ whose points are flags of vector spaces of odd dimensions. 
Its Schubert varieties $\widetilde{X}_w$ are indexed by minimal length coset representatives $w\in S_{2n-2}/W_{P}$, where $W_P$ is the Weyl group of the Levi of $P$. 
More precisely, if $\widetilde{s_i}\in S_{2n-2}$ denotes the simple transposition $(i,i+1)$, then $W_P=\langle \widetilde{s}_2,\widetilde{s}_4,\ldots\widetilde{s}_{2n-4}\rangle$.
Let $w_n\in S_{2n-2}$ be defined by
\[
w_n(i)=\left\{
\begin{array}{ll}
r & \hbox{ if }i=2r,r\ge 1,\\
n+r-1 & \hbox{ if } i=2r-1,r\in[n-1].
\end{array}
\right.
\]
The following Theorem can be found in \cite{C-IL15}.
\begin{theorem}[\cite{C-IL15}]\label{thm:CI-L}
The degenerate flag variety $\Flag_n^a$ is isomorphic to the Schubert variety $\widetilde{X}_{w_n}\subset SL_{2n-2}/P$. 
\end{theorem}

\subsubsection{Translation into Pl\"ucker coordinates}We describe here the isomorphism of Theorem \ref{thm:CI-L}
 in terms of Pl\"ucker coordinates. Recall that whenever we index Pl\"ucker coordinates by a set, we really mean the associated sequence obtained by increasingly ordering the elements of the given set.
 
Let $J\in \binom{[2n-2]}{2k-1}$, with $k\in[n-1]$, then $J\leq w_n([2k-1])=[k-1]\cup [n,n+k-1]$ if and only if
\begin{equation}\label{eqn:richardson1}
[k-1]\subset J\subset [k+n-1].
\end{equation}
 In order to give the translation of the isomorphism in terms of coordinate rings, we  need to set some notation.
Let $k\in[n-1]$, we denote by $\{\leq w_n\}^{(2k-1)}$ the set of $J\in \binom{[2n-2]}{2k-1}$, with $J\leq w_n([2k-1])$. There is hence a bijection
\begin{equation}\label{eqn:bij1}
\{\leq w_n\}^{(2k-1)}\rightarrow \binom{[n]}{k}, \quad  J\mapsto \tau_k(J\setminus [k-1])
\end{equation}
where $\tau_k:[n+k-1]\rightarrow[n]$ is given  by
\[
\tau_k(j)\mapsto\left\{
\begin{array}{ll}
j&\hbox{if }j\in[k,n],\\
j-n&\hbox{if }j\in[n+1,n+k-1].\\
\end{array}
\right.
\]
For a sequence $I=(i_1, \ldots, i_{k})\in\mathcal{S}(n,k)$ we set $\tau_k(I):=(\tau_k(i_1), \ldots \tau_k(i_k))\in\mathcal{S}(n,k)$.
If $\rho_k:[n]\rightarrow[k,n+k-1]$ is given  by
\[
\rho_k(j)\mapsto\left\{
\begin{array}{ll}
j&\hbox{if }j\in[k,n],\\
j+n&\hbox{if }j\in[k-1],\\
\end{array}
\right.
\]
then the inverse map to \eqref{eqn:bij1} is given by 
\[
\binom{[n]}{k} \rightarrow\{\leq w_n\}^{(2k-1)}, \quad  I\mapsto [k-1]\cup \rho_k(I).
\]
On the level of sequences, this lifts to a map
\[
\begin{array}{ccc}
\mathcal{S}(n,k)&\stackrel{\widetilde{\rho_k}}{\rightarrow}&\left\{J\in\mathcal{S}(2n-2,2k-1)\mid F(J)\in \{\leq w_n\}^{(2k-1)}\right\},\\
(i_1, \ldots, i_{k})&\mapsto& (1,2,\ldots,k-1,\rho_k(i_1), \ldots, \rho_k(i_k))
\end{array}\]

Fix an ordered basis $(\tilde{e}_j)_{j\in[2n-2]}$ of $\mathbb C^{2n-2}$, then the linear algebraic description of $\widetilde{X}_{w_n}$ is
\[
\widetilde{X}_{w_n}=\left\{
\{0\}\subset W_1\subset W_3\subset\ldots\subset W_{2n-3}\left|
\begin{array}{c}
W_{2k-1}\in\textrm{Gr}(2k-1,\mathbb C^{2n-2})\\
\textrm{span}_{\mathbb{C}}\{\widetilde{e}_j\mid j\in[k-1]\}\subset W_{2k-1},\\
W_{2k-1}\subset \textrm{span}_{\mathbb{C}}\{\widetilde{e}_j\mid j\in[n+k-1]\}.
\end{array}\right.
\right\}
\]
Denote by $(e_i)_{i\in[n]}$ an ordered basis for $\mathbb{C}^n$. 
For $k\in[n-1]$ define
the projection operator (which we also denote by $\pi_k$ as in \cite{C-IL15})
\[
\begin{array}{rrcc}
\pi_{k}&:\textrm{span}_{\mathbb{C}}\{\widetilde{e}_j\mid j\in[n+k-1]\}&\rightarrow&\mathbb{C}^n=\textrm{span}_{\mathbb{C}}\{e_i\mid i\in[n]\},\\
& \widetilde{e}_j&\mapsto& 
\left\{
\begin{array}{ll}
e_{\tau_k(j)}&\hbox{if }j\in[k,n+k-1],\\
0&\hbox{otherwise}
\end{array}
\right.
\end{array}\]
Then there is an isomorphism, which we denote by the same symbol, of algebraic varieties
\begin{eqnarray*}
\widetilde{X}_{w_n}^{(2k-1)}:=\left\{
U \left|
\begin{array}{c}
U\in\textrm{Gr}(2k-1,\mathbb C^{2n-2})\\
\textrm{span}_{\mathbb{C}}\{\widetilde{e}_j\mid j\in[2i-2]\}\subset U,\\
U\subset \textrm{span}_{\mathbb{C}}\{\widetilde{e}_j\mid j\in[n+2k-2]\}.
\end{array}\right.
\right\}
&\stackrel{\pi_{k}}{\longrightarrow}& \textrm{Gr}(k,\mathbb C^n),\\
 U &\mapsto& \pi_{k}(U)
\end{eqnarray*}
and the desired isomorphism (cf. \cite{C-IL15})  is given by
\begin{equation}\label{eqn:isoCIL15}
\xi:\widetilde{X}_{w_n}\rightarrow\Flag_n^a,\quad (W_{2k-1})_{k\in[n-1]}\mapsto (\pi_{k}(W_{2k-1}))_{k\in[n-1]}.
\end{equation}

 \begin{remark}In \cite{C-IL15}, an embedding of $\zeta:\Flag_n\hookrightarrow SL_{2n-2}/P$ is given, and hence the isomorphism from Theorem \ref{thm:CI-L} is rather the inverse of the isomorphism $\xi$ we consider here. We prefer to work with $\xi$ instead of $\zeta$ since in this way we obtain an induced map from the coordinate ring of $\Flag_n^a$ to the coordinate ring of $\widetilde{X}_{w_n}$, which we make explicit in the following.
 \end{remark}

For $SL_{2n-2}/P$ we also have an embedding into the product of Grassmannians
\[
SL_{2n-2}/P\hookrightarrow \Gr(1,\mathbb C^{2n-2})\times Gr(3,\mathbb C^{2n-2})\times\dots\times\Gr(2n-3,\mathbb C^{2n-2}),
\]
and hence a Pl\"ucker embedding.
Pl\"ucker coordinates for $\Gr(2k-1,\mathbb C^{2n-2})$ with $k\in[n-1]$ are denoted by $\tilde p_J, J\in\mathcal{S}(2n-2, 2k-1)$. Let $I=(i_1,\dots,i_k)$ then 
\[
\pi_k^*:\mathbb{C}[\textrm{Gr}(k,n)]\rightarrow \mathbb{C}[\widetilde{X}_w^{(2k-1)}], \quad p_I\mapsto \widetilde{p}_{\widetilde{\rho}_k(I)} .
\]
As $\pi_k^*$ is compatible with Pl\"ucker relations, we have an isomorphism
\[
\xi^*:\mathbb{C}[\Flag_n^a]\rightarrow\mathbb{C}[\widetilde{X}_{w_n}], \quad p_I\mapsto \pi^*_{\#I}(p_I).
\]

Notice that even if $I$ is ordered increasingly, $\widetilde{\rho_k}(I)$ needs not be ordered increasingly. To get an increasing sequence we have to multiply by some sign. While keeping track of the sign is fundamental to check that Pl\"ucker relations are satisfied, it is not relevant to us, as we only deal with vanishing of certain Pl\"ucker coordinates, which of course vanish independently of their sign.

\subsection{Richardson varieties in $SL_{2n-2}/P$}
Let $u,v\in S_{2n-2}$ be minimal length coset representatives of $S_{2n-2}/W_P$ and assume that $u\leq v$. We denote by $\widetilde{X}_v^u:=\widetilde{X}_v\cap \widetilde{X}^u\subseteq SL_{2n-2}/P$ the corresponding Richardson variety. Recall that its defining ideal in $\mathbb{C}[p_I\mid\#I\equiv 1\textrm{\small{(mod 2)}},I\subset[2n-2]]$ is
\begin{equation}\label{eq:idealRichardson}
\mathcal{I}_v^u=(R^k_{J,L})+(p_I)_{I\not\le v([\#I])}+(p_I)_{I\not\ge u([\#I])}.
\end{equation}

In the following we will show that for appropriate permutations $x\in S_n$, $u,v\in S_{2n-2}$ with $u\leq v\leq w_n$, 
the isomorphism $\xi^*$ induces an isomorphism between the coordinate rings
\[
\mathbb{C}[X^a_x]\rightarrow \mathbb{C}[\widetilde{X}^u_v].\]
To stress out the fact that such an isomorphism really comes from the embedding $\zeta$, we will express it as $\zeta(X^a_x)=\widetilde{X}_v^u$.

Since $\mathbb{C}[X^a_x]=
\mathbb{C}[\Flag_n^a]/(p_I\mid I\nleq x([\#I]))$ and $\mathbb{C}[\widetilde{X}^u_v]=\mathbb{C}[SL_{2n-2}/P]/(p_K\mid K\nleq v([\#K]), \ K\ngeq u([\#K]))$, the claim will be proven by verifying that
\begin{equation}\label{eqn:richardson2}
\left(
(K\leq v([\#K])\hbox { and }K\geq u([\#K])
\right)
  \quad\Rightarrow\quad 
  \tau_{k}(K\setminus[k-1])  \leq x([k]),
\end{equation}
where $k:=\frac{\#K+1}{2}$,
and the opposite direction
\begin{equation}\label{eqn:richardson3}
I\leq x(\#I)\quad\Rightarrow \quad
\left(
\begin{array}{c}
[k-1]\cup \rho_{\#I}(I)\leq v([n-1+\#I])\\
{[k-1]\cup \rho_{\#I}(I)\geq u([n-1+\#I])}
\end{array}
\right).  
\end{equation}

An important role will be played by the following permutation $y_n\in S_{2n-2}$:
\[
y_n(i)=\left\{
\begin{matrix}
1 & \text{ if } i=1,\\
r+1 & \text{ if } i=2r,r\in[n-1],\\
n+r-1 & \text{ if } i=2r-1,r\in[n-1].
\end{matrix}
\right.
\]
Notice that for any $m\in[n-1]$
\[
\tilde{s}_m\tilde{s}_{m-1}\ldots \tilde{s}_1y_n(i)=
\left\{
\begin{matrix}
m+1 & \text{ if } i=1,\\
r & \text{ if } i=2r,\ r\in[m],\\
r+1& \text{ if } i=2r, \ r\in[m+1,n-1],\\
n+r-1 & \text{ if } i=2r-1,r\in[n-1],
\end{matrix}
\right.
\]
and, by \eqref{eq:BBequiv}, $y_n< \tilde{s}_m\tilde{s}_{m-1}\ldots \tilde{s}_1y_n\leq w_n$.

\begin{lemma}\label{lem:richardson}
Let $m\in[n-1]$ and $x:=s_{m}s_{m-1}\dots s_1\in S_n$. Then, \[\zeta(X_{x}^a)=\widetilde{X}_{\tilde{s}_m\tilde{s}_{m-1}\ldots \tilde{s}_1y_n}^{y_n}.\]
\end{lemma}
\begin{proof} 
Let $I\in\binom{[n]}{k}$. Then, by \eqref{eq:coxeter}, $I\leq x([k])$ if and only if
\begin{eqnarray*}
I=\left\{ \begin{matrix} 
[k-1]\cup\{i\}, i\in[k,m+1] & \text{ if } k\le m,\\
[k] & \text{ if } k>m.
\end{matrix} \right.
\end{eqnarray*}
On the other hand, let $K\in \binom{2n-2}{2k-1}$, then both $K\leq \tilde{s}_m\tilde{s}_{m-1}\ldots \tilde{s}_1y_n([2k-1])$  and $K\geq y_n([2k-1])$ hold if and only if 
\begin{eqnarray*}
K=\left\{\begin{matrix}
[k-1]\cup[n+1,n+k-1]\cup\{i\},i\in[k,m+1] & \text{ if } k\le m, \\
[k]\cup [n+1,n+k-1] & \text{ if } k>m.
\end{matrix}\right.
\end{eqnarray*}
These two facts imply \eqref{eqn:richardson2} and \eqref{eqn:richardson3}.
\end{proof}

Combining Lemma~\ref{lem:richardson} with Proposition~\ref{prop:isomschubert} we obtain the following corollary.
\begin{corollary}
Let $x=s_{m}s_{m-1}\cdots s_1\le c$ and consider the Schubert variety $X_x\subset \Flag_n$. Then there is an isomorphism
\[
X_v\cong\widetilde{X}_{\tilde{s}_{m}\tilde{s}_{m-1}\ldots\tilde{s}_{1}y_n}^{y_n}\subset SL_{2n-2}/P.
\]
\end{corollary}

\section{Schubert divisors}\label{sec:codim1}
In this section we focus on Schubert divisors and apply the results from previous sections to them.
In this case we can completely answer the question whether or not they stay irreducible under the degeneration. 

Let $w_0\in S_n$ be the longest element, then all Schubert divisors are indexed by permutations of the form $w=w_0s_i$ for $i\in[n-1]$. Note that
\begin{equation*}
w(k)=\left\{
\begin{array}{ll}
n-k+1&\hbox{ if }k\neq i,i+1,\\
n-i&\hbox{ if }k=i,\\
n-i+1&\hbox{ if }k= i+1.
\end{array}
\right.
\end{equation*}

The following Theorem~\ref{cor:redcodim1} is an application of Theorem~\ref{thm:alllemma}~(1) and (2).

\begin{theorem}\label{cor:redcodim1}
Let $n>2$ and $w\in S_n$ be such that $ws_i=w_0$. If $n$ is odd assume $i\not=\frac{n+1}{2}$, for even $n$ there is no additional assumption. Then $X^a_w$ is reducible.
\end{theorem}
\begin{proof}
We consider four cases separately: $i<\frac{n}{2}$, $i=\frac{n}{2}$, $i\geq \frac{n+3}{2}$, and $i=\frac{n+2}{2}$. Notice that they cover all possiblities, since $i> \frac{n}{2}$ together with the assumption  $i\neq\frac{n+1}{2}$ implies  $i>\frac{n+1}{2}$, hence $i\geq \frac{n+2}{2}$. We will deal with the first two cases by applying Theorem~\ref{thm:alllemma}~(1), while we will use Theorem~\ref{thm:alllemma}~(2) for the remaining two.

First of all, notice that $w_0=ws_i>w$. 

\noindent
\underline{Case 1:} If $i<\frac{n}{2}$, then 
\begin{equation}\label{eqn:redcodim1a}
w(k)=n-k+1\ge n-i+2>\frac{n}{2}+2>i, \quad \hbox{ for any }k\leq i-1,
\end{equation}
\[w(i)=n-i>\frac{n}{2}>i,\]
and
\begin{equation}\label{eqn:redcodim1b}
w(i+1)=n-i+1>\frac{n}{2}+1>i.
\end{equation} 
We conclude that $i\not\in w([i+1])$ and we can hence apply Theorem~\ref{thm:alllemma}~(1) with $j=w(i)$.

\noindent
\underline{Case 2:} If $i=\frac{n}{2}$, then \eqref{eqn:redcodim1a} and \eqref{eqn:redcodim1b} still hold, but $w(i)=i$, so that $i\not\in w([i-1])\cup\{w(i+1)\}$, but we cannot choose $j=w(i)$.  
Nevertheless, \eqref{eqn:redcodim1a} and \eqref{eqn:redcodim1b} imply that  any $j$ with $j\leq i-1<i=w(i)$ (which exists, since $n>2$) fulfills the hypotheses of Theorem~\ref{thm:alllemma}~(1).

\noindent
\underline{Case 3:}
Let $i\geq\frac{n+3}{2}$, so that $n\leq 2i-3$ and 
$n-i+2\leq  2i-3-i+2=i-1$. Note further that $w(i+1)=n+i-1\le \frac{n+3}{2}-1\le i-1$. Thus $w(n-i+2)=i-1\in w([i-1])$ and we can apply Theorem~\ref{thm:alllemma}~(2) with $l=w(i)$ and $x=w(i+1)$.

\noindent
\underline{Case 4:} Consider $i=\frac{n}{2}+1$. In this case, $w(i+1)=n-i+1=\frac{n}{2}=i-1\in w([i-1])\cup \{w(i+1)\}$ and we can apply Theorem~\ref{thm:alllemma}~(2) with $x$ any element in $w([i-1])$ and $l=w(i)$.
\end{proof}

For flag varieties $\Flag_n$ with $n$ odd, the next proposition explains why the case of $w_0s_i$ for $i=\frac{n+1}{2}$ is special. This is another instance, of a degenerate Schubert variety being isomorphic to a Richardson variety in $SL_{2n-2}/P$. However,  unlike the degenerate Schubert varieties of form $X^a_v$, for $v\leq c$,  this one is not isomorphic to the original Schubert variety.

\begin{proposition}\label{prop:oddrichardson}
Let $i\geq 2$ and $n=2i-1$. Then
$\zeta(X_{w_0s_{i}}^a)= \widetilde{X}_{w_n}^{\widetilde{s}_{2i-1}}$.
\end{proposition}
\begin{proof}
First note that $w_0s_i([i])=\{n-i\}\cup[n-i+2,n]=\{i-1\}\cup[i+1,n]$ and $w_0([i])=[n-i+1,n]=[i,n]$. 
Let $J\in\binom{n}{k}$, then $J\nleq w_0s_{i}([k])=[n-k+1,n]$ if and only if $k=i$ and $J=[i,n]$.

On the other hand, recall that $w_n([2k-1])=[k-1]\cup[n+k-1,n]$ and
\[
\tilde{s}_{2i-1}([2k-1])=\left\{\begin{matrix}
[2k-1] & \text { if } k\not =i,\\
[2i-2]\cup \{2i\} & \text{ if } k=i.
\end{matrix}\right.
\]
If $K\in\binom{2n-2}{2k-1}$ is such that $K\leq w_n([2k-1])$, then $K\ngeq \tilde{s}_{2i-1}([2k-1])$ if and only if $k=i$ and $K=[2i-1]=[n]$.

At this point the claim follows from $\pi_i^*(p_{[i,n]})=\tilde{p}_{[i-1]\cup\rho_i([i,n])}=\tilde{p}_{[n]}$. 
\end{proof}

\begin{corollary} 
\begin{enumerate}
\item If $n$ is even, then all Schubert divisors $X_{w_0s_i}\subset \Flag_n$ become reducible under Feigin's degeneration.
\item If $n$ is odd, then the Schubert divisor $X_{w_0s_{\frac{n+1}{2}}}\subset \Flag_n$ stays irreducible under Feigin's degeneration, while all the others become reducible.
\end{enumerate}
\end{corollary}

\bibliographystyle{alpha} 
\bibliography{Trop.bib}

\section*{Appendix}
Table~\ref{tab:S5} shows which of the criteria for $\init_{\bf w}(\mathcal{I}_w)$ to contain a monomial apply to which elements $w\in S_5$. 
It has to be read as follows: the first column contains $w\in S_5$ written in one-line notation, the second  as a reduced word. 
In the third column ``$\times$" indicates that $\init_{\bf w}(\mathcal{I}_w)$ contains a monomial, resp. ``$-$" that it does not. The last columns labeled (1) to (7) indicate which of the points of Theorem~\ref{thm:alllemma} apply to $w$.
The last row indicates how often $\times$ appears in the corresponding column.

{
\footnotesize
\begin{longtable}{@{\extracolsep{\fill}}ll | c|lllllll@{}}
$w$ one-line & $w$ reduced word & mono. & (1) & (2) & (3) & (4)& (5) & (6) & (7) \\ \hline \endhead
$[1, 2, 3, 4, 5]$ & $1$ & $-$ & $-$ & $-$ & $-$ & $-$ & $-$ & $-$ & $-$ \\
$[1, 2, 3, 5, 4]$ & $s_{4}$ & $-$ & $-$ & $-$ & $-$ & $-$ & $-$ & $-$ & $-$ \\
$[1, 2, 4, 3, 5]$ & $s_{3}$ & $-$ & $-$ & $-$ & $-$ & $-$ & $-$ & $-$ & $-$ \\
$[1, 2, 4, 5, 3]$ & $s_{3}s_{4}$ & $\times$ & $\times$ & $-$ & $\times$ & $-$ & $-$ & $\times$ & $-$ \\
$[1, 2, 5, 3, 4]$ & $s_{4}s_{3}$ & $-$ & $-$ & $-$ & $-$ & $-$ & $-$ & $-$ & $-$ \\
$[1, 2, 5, 4, 3]$ & $s_{3}s_{4}s_{3}$ & $-$ & $-$ & $-$ & $-$ & $-$ & $-$ & $-$ & $-$ \\
$[1, 3, 2, 4, 5]$ & $s_{2}$ & $-$ & $-$ & $-$ & $-$ & $-$ & $-$ & $-$ & $-$ \\
$[1, 3, 2, 5, 4]$ & $s_{4}s_{2}$ & $-$ & $-$ & $-$ & $-$ & $-$ & $-$ & $-$ & $-$ \\
$[1, 3, 4, 2, 5]$ & $s_{2}s_{3}$ & $\times$ & $\times$ & $-$ & $\times$ & $-$ & $-$ & $\times$ & $-$ \\
$[1, 3, 4, 5, 2]$ & $s_{2}s_{3}s_{4}$ & $\times$ & $\times$ & $-$ & $\times$ & $\times$ & $-$ & $\times$ & $-$ \\
$[1, 3, 5, 2, 4]$ & $s_{4}s_{2}s_{3}$ & $\times$ & $\times$ & $-$ & $\times$ & $-$ & $-$ & $\times$ & $-$ \\
$[1, 3, 5, 4, 2]$ & $s_{2}s_{3}s_{4}s_{3}$ & $\times$ & $\times$ & $-$ & $\times$ & $\times$ & $-$ & $\times$ & $-$ \\
$[1, 4, 2, 3, 5]$ & $s_{3}s_{2}$ & $-$ & $-$ & $-$ & $-$ & $-$ & $-$ & $-$ & $-$ \\
$[1, 4, 2, 5, 3]$ & $s_{3}s_{4}s_{2}$ & $\times$ & $-$ & $-$ & $\times$ & $-$ & $-$ & $\times$ & $-$ \\
$[1, 4, 3, 2, 5]$ & $s_{2}s_{3}s_{2}$ & $-$ & $-$ & $-$ & $-$ & $-$ & $-$ & $-$ & $-$ \\
$[1, 4, 3, 5, 2]$ & $s_{2}s_{3}s_{4}s_{2}$ & $\times$ & $\times$ & $-$ & $-$ & $\times$ & $-$ & $\times$ & $-$ \\
$[1, 4, 5, 2, 3]$ & $s_{3}s_{4}s_{2}s_{3}$ & $\times$ & $\times$ & $\times$ & $\times$ & $-$ & $\times$ & $\times$ & $-$ \\
$[1, 4, 5, 3, 2]$ & $s_{2}s_{3}s_{4}s_{2}s_{3}$ & $\times$ & $\times$ & $-$ & $\times$ & $-$ & $-$ & $\times$ & $-$ \\
$[1, 5, 2, 3, 4]$ & $s_{4}s_{3}s_{2}$ & $-$ & $-$ & $-$ & $-$ & $-$ & $-$ & $-$ & $-$ \\
$[1, 5, 2, 4, 3]$ & $s_{3}s_{4}s_{3}s_{2}$ & $-$ & $-$ & $-$ & $-$ & $-$ & $-$ & $-$ & $-$ \\
$[1, 5, 3, 2, 4]$ & $s_{4}s_{2}s_{3}s_{2}$ & $-$ & $-$ & $-$ & $-$ & $-$ & $-$ & $-$ & $-$ \\
$[1, 5, 3, 4, 2]$ & $s_{2}s_{3}s_{4}s_{3}s_{2}$ & $\times$ & $\times$ & $-$ & $-$ & $\times$ & $-$ & $-$ & $\times$ \\
$[1, 5, 4, 2, 3]$ & $s_{3}s_{4}s_{2}s_{3}s_{2}$ & $\times$ & $-$ & $\times$ & $-$ & $-$ & $\times$ & $-$ & $-$ \\
$[1, 5, 4, 3, 2]$ & $s_{2}s_{3}s_{4}s_{2}s_{3}s_{2}$ & $-$ & $-$ & $-$ & $-$ & $-$ & $-$ & $-$ & $-$ \\
$[2, 1, 3, 4, 5]$ & $s_{1}$ & $-$ & $-$ & $-$ & $-$ & $-$ & $-$ & $-$ & $-$ \\
$[2, 1, 3, 5, 4]$ & $s_{4}s_{1}$ & $-$ & $-$ & $-$ & $-$ & $-$ & $-$ & $-$ & $-$ \\
$[2, 1, 4, 3, 5]$ & $s_{3}s_{1}$ & $-$ & $-$ & $-$ & $-$ & $-$ & $-$ & $-$ & $-$ \\
$[2, 1, 4, 5, 3]$ & $s_{3}s_{4}s_{1}$ & $\times$ & $\times$ & $-$ & $\times$ & $-$ & $-$ & $-$ & $-$ \\
$[2, 1, 5, 3, 4]$ & $s_{4}s_{3}s_{1}$ & $-$ & $-$ & $-$ & $-$ & $-$ & $-$ & $-$ & $-$ \\
$[2, 1, 5, 4, 3]$ & $s_{3}s_{4}s_{3}s_{1}$ & $-$ & $-$ & $-$ & $-$ & $-$ & $-$ & $-$ & $-$ \\
$[2, 3, 1, 4, 5]$ & $s_{1}s_{2}$ & $\times$ & $\times$ & $-$ & $\times$ & $-$ & $-$ & $\times$ & $-$ \\
$[2, 3, 1, 5, 4]$ & $s_{4}s_{1}s_{2}$ & $\times$ & $\times$ & $-$ & $\times$ & $-$ & $-$ & $\times$ & $-$ \\
$[2, 3, 4, 1, 5]$ & $s_{1}s_{2}s_{3}$ & $\times$ & $\times$ & $-$ & $\times$ & $\times$ & $-$ & $\times$ & $-$ \\
$[2, 3, 4, 5, 1]$ & $s_{1}s_{2}s_{3}s_{4}$ & $\times$ & $\times$ & $-$ & $\times$ & $\times$ & $-$ & $\times$ & $-$ \\
$[2, 3, 5, 1, 4]$ & $s_{4}s_{1}s_{2}s_{3}$ & $\times$ & $\times$ & $-$ & $\times$ & $\times$ & $-$ & $\times$ & $-$ \\
$[2, 3, 5, 4, 1]$ & $s_{1}s_{2}s_{3}s_{4}s_{3}$ & $\times$ & $\times$ & $-$ & $\times$ & $\times$ & $-$ & $\times$ & $-$ \\
$[2, 4, 1, 3, 5]$ & $s_{3}s_{1}s_{2}$ & $\times$ & $\times$ & $-$ & $\times$ & $-$ & $-$ & $\times$ & $-$ \\
$[2, 4, 1, 5, 3]$ & $s_{3}s_{4}s_{1}s_{2}$ & $\times$ & $\times$ & $-$ & $\times$ & $-$ & $-$ & $\times$ & $-$ \\
$[2, 4, 3, 1, 5]$ & $s_{1}s_{2}s_{3}s_{2}$ & $\times$ & $\times$ & $-$ & $\times$ & $\times$ & $-$ & $\times$ & $-$ \\
$[2, 4, 3, 5, 1]$ & $s_{1}s_{2}s_{3}s_{4}s_{2}$ & $\times$ & $\times$ & $-$ & $\times$ & $\times$ & $-$ & $\times$ & $-$ \\
$[2, 4, 5, 1, 3]$ & $s_{3}s_{4}s_{1}s_{2}s_{3}$ & $\times$ & $\times$ & $\times$ & $\times$ & $\times$ & $\times$ & $\times$ & $-$ \\
$[2, 4, 5, 3, 1]$ & $s_{1}s_{2}s_{3}s_{4}s_{2}s_{3}$ & $\times$ & $\times$ & $-$ & $\times$ & $\times$ & $-$ & $\times$ & $-$ \\
$[2, 5, 1, 3, 4]$ & $s_{4}s_{3}s_{1}s_{2}$ & $\times$ & $\times$ & $-$ & $\times$ & $-$ & $-$ & $\times$ & $-$ \\
$[2, 5, 1, 4, 3]$ & $s_{3}s_{4}s_{3}s_{1}s_{2}$ & $\times$ & $\times$ & $-$ & $\times$ & $-$ & $-$ & $\times$ & $-$ \\
$[2, 5, 3, 1, 4]$ & $s_{4}s_{1}s_{2}s_{3}s_{2}$ & $\times$ & $\times$ & $-$ & $\times$ & $\times$ & $-$ & $\times$ & $-$ \\
$[2, 5, 3, 4, 1]$ & $s_{1}s_{2}s_{3}s_{4}s_{3}s_{2}$ & $\times$ & $\times$ & $-$ & $\times$ & $\times$ & $-$ & $\times$ & $-$ \\
$[2, 5, 4, 1, 3]$ & $s_{3}s_{4}s_{1}s_{2}s_{3}s_{2}$ & $\times$ & $\times$ & $\times$ & $\times$ & $\times$ & $\times$ & $\times$ & $-$ \\
$[2, 5, 4, 3, 1]$ & $s_{1}s_{2}s_{3}s_{4}s_{2}s_{3}s_{2}$ & $\times$ & $\times$ & $-$ & $\times$ & $\times$ & $-$ & $\times$ & $-$ \\
$[3, 1, 2, 4, 5]$ & $s_{2}s_{1}$ & $-$ & $-$ & $-$ & $-$ & $-$ & $-$ & $-$ & $-$ \\
$[3, 1, 2, 5, 4]$ & $s_{4}s_{2}s_{1}$ & $-$ & $-$ & $-$ & $-$ & $-$ & $-$ & $-$ & $-$ \\
$[3, 1, 4, 2, 5]$ & $s_{2}s_{3}s_{1}$ & $\times$ & $-$ & $-$ & $\times$ & $-$ & $-$ & $\times$ & $-$ \\
$[3, 1, 4, 5, 2]$ & $s_{2}s_{3}s_{4}s_{1}$ & $\times$ & $-$ & $-$ & $\times$ & $\times$ & $-$ & $\times$ & $-$ \\
$[3, 1, 5, 2, 4]$ & $s_{4}s_{2}s_{3}s_{1}$ & $\times$ & $-$ & $-$ & $\times$ & $-$ & $-$ & $\times$ & $-$ \\
$[3, 1, 5, 4, 2]$ & $s_{2}s_{3}s_{4}s_{3}s_{1}$ & $\times$ & $-$ & $-$ & $\times$ & $\times$ & $-$ & $\times$ & $-$ \\
$[3, 2, 1, 4, 5]$ & $s_{1}s_{2}s_{1}$ & $-$ & $-$ & $-$ & $-$ & $-$ & $-$ & $-$ & $-$ \\
$[3, 2, 1, 5, 4]$ & $s_{4}s_{1}s_{2}s_{1}$ & $-$ & $-$ & $-$ & $-$ & $-$ & $-$ & $-$ & $-$ \\
$[3, 2, 4, 1, 5]$ & $s_{1}s_{2}s_{3}s_{1}$ & $\times$ & $\times$ & $-$ & $-$ & $\times$ & $-$ & $\times$ & $-$ \\
$[3, 2, 4, 5, 1]$ & $s_{1}s_{2}s_{3}s_{4}s_{1}$ & $\times$ & $\times$ & $-$ & $-$ & $\times$ & $-$ & $\times$ & $-$ \\
$[3, 2, 5, 1, 4]$ & $s_{4}s_{1}s_{2}s_{3}s_{1}$ & $\times$ & $\times$ & $-$ & $-$ & $\times$ & $-$ & $\times$ & $-$ \\
$[3, 2, 5, 4, 1]$ & $s_{1}s_{2}s_{3}s_{4}s_{3}s_{1}$ & $\times$ & $\times$ & $-$ & $-$ & $\times$ & $-$ & $\times$ & $-$ \\
$[3, 4, 1, 2, 5]$ & $s_{2}s_{3}s_{1}s_{2}$ & $\times$ & $\times$ & $\times$ & $\times$ & $-$ & $\times$ & $\times$ & $-$ \\
$[3, 4, 1, 5, 2]$ & $s_{2}s_{3}s_{4}s_{1}s_{2}$ & $\times$ & $\times$ & $-$ & $\times$ & $\times$ & $\times$ & $\times$ & $-$ \\
$[3, 4, 2, 1, 5]$ & $s_{1}s_{2}s_{3}s_{1}s_{2}$ & $\times$ & $\times$ & $-$ & $\times$ & $-$ & $-$ & $\times$ & $-$ \\
$[3, 4, 2, 5, 1]$ & $s_{1}s_{2}s_{3}s_{4}s_{1}s_{2}$ & $\times$ & $\times$ & $-$ & $\times$ & $-$ & $-$ & $\times$ & $-$ \\
$[3, 4, 5, 1, 2]$ & $s_{2}s_{3}s_{4}s_{1}s_{2}s_{3}$ & $\times$ & $\times$ & $\times$ & $\times$ & $-$ & $-$ & $\times$ & $-$ \\
$[3, 4, 5, 2, 1]$ & $s_{1}s_{2}s_{3}s_{4}s_{1}s_{2}s_{3}$ & $\times$ & $\times$ & $-$ & $\times$ & $-$ & $-$ & $\times$ & $-$ \\
$[3, 5, 1, 2, 4]$ & $s_{4}s_{2}s_{3}s_{1}s_{2}$ & $\times$ & $\times$ & $\times$ & $\times$ & $-$ & $\times$ & $\times$ & $-$ \\
$[3, 5, 1, 4, 2]$ & $s_{2}s_{3}s_{4}s_{3}s_{1}s_{2}$ & $\times$ & $\times$ & $-$ & $\times$ & $\times$ & $\times$ & $\times$ & $-$ \\
$[3, 5, 2, 1, 4]$ & $s_{4}s_{1}s_{2}s_{3}s_{1}s_{2}$ & $\times$ & $\times$ & $-$ & $\times$ & $-$ & $-$ & $\times$ & $-$ \\
$[3, 5, 2, 4, 1]$ & $s_{1}s_{2}s_{3}s_{4}s_{3}s_{1}s_{2}$ & $\times$ & $\times$ & $-$ & $\times$ & $-$ & $-$ & $\times$ & $-$ \\
$[3, 5, 4, 1, 2]$ & $s_{2}s_{3}s_{4}s_{1}s_{2}s_{3}s_{2}$ & $\times$ & $\times$ & $\times$ & $\times$ & $-$ & $-$ & $\times$ & $-$ \\
$[3, 5, 4, 2, 1]$ & $s_{1}s_{2}s_{3}s_{4}s_{1}s_{2}s_{3}s_{2}$ & $\times$ & $\times$ & $-$ & $\times$ & $-$ & $-$ & $\times$ & $-$ \\
$[4, 1, 2, 3, 5]$ & $s_{3}s_{2}s_{1}$ & $-$ & $-$ & $-$ & $-$ & $-$ & $-$ & $-$ & $-$ \\
$[4, 1, 2, 5, 3]$ & $s_{3}s_{4}s_{2}s_{1}$ & $\times$ & $-$ & $-$ & $\times$ & $-$ & $-$ & $\times$ & $-$ \\
$[4, 1, 3, 2, 5]$ & $s_{2}s_{3}s_{2}s_{1}$ & $-$ & $-$ & $-$ & $-$ & $-$ & $-$ & $-$ & $-$ \\
$[4, 1, 3, 5, 2]$ & $s_{2}s_{3}s_{4}s_{2}s_{1}$ & $\times$ & $\times$ & $-$ & $-$ & $\times$ & $-$ & $\times$ & $-$ \\
$[4, 1, 5, 2, 3]$ & $s_{3}s_{4}s_{2}s_{3}s_{1}$ & $\times$ & $-$ & $\times$ & $\times$ & $-$ & $\times$ & $\times$ & $-$ \\
$[4, 1, 5, 3, 2]$ & $s_{2}s_{3}s_{4}s_{2}s_{3}s_{1}$ & $\times$ & $-$ & $-$ & $\times$ & $-$ & $-$ & $\times$ & $-$ \\
$[4, 2, 1, 3, 5]$ & $s_{3}s_{1}s_{2}s_{1}$ & $-$ & $-$ & $-$ & $-$ & $-$ & $-$ & $-$ & $-$ \\
$[4, 2, 1, 5, 3]$ & $s_{3}s_{4}s_{1}s_{2}s_{1}$ & $\times$ & $-$ & $-$ & $\times$ & $-$ & $-$ & $\times$ & $-$ \\
$[4, 2, 3, 1, 5]$ & $s_{1}s_{2}s_{3}s_{2}s_{1}$ & $\times$ & $\times$ & $-$ & $-$ & $\times$ & $-$ & $-$ & $-$ \\
$[4, 2, 3, 5, 1]$ & $s_{1}s_{2}s_{3}s_{4}s_{2}s_{1}$ & $\times$ & $\times$ & $-$ & $-$ & $\times$ & $-$ & $\times$ & $-$ \\
$[4, 2, 5, 1, 3]$ & $s_{3}s_{4}s_{1}s_{2}s_{3}s_{1}$ & $\times$ & $\times$ & $\times$ & $-$ & $\times$ & $\times$ & $\times$ & $-$ \\
$[4, 2, 5, 3, 1]$ & $s_{1}s_{2}s_{3}s_{4}s_{2}s_{3}s_{1}$ & $\times$ & $\times$ & $-$ & $-$ & $\times$ & $-$ & $\times$ & $-$ \\
$[4, 3, 1, 2, 5]$ & $s_{2}s_{3}s_{1}s_{2}s_{1}$ & $\times$ & $-$ & $\times$ & $-$ & $-$ & $\times$ & $-$ & $-$ \\
$[4, 3, 1, 5, 2]$ & $s_{2}s_{3}s_{4}s_{1}s_{2}s_{1}$ & $\times$ & $-$ & $-$ & $-$ & $\times$ & $\times$ & $\times$ & $-$ \\
$[4, 3, 2, 1, 5]$ & $s_{1}s_{2}s_{3}s_{1}s_{2}s_{1}$ & $-$ & $-$ & $-$ & $-$ & $-$ & $-$ & $-$ & $-$ \\
$[4, 3, 2, 5, 1]$ & $s_{1}s_{2}s_{3}s_{4}s_{1}s_{2}s_{1}$ & $\times$ & $-$ & $-$ & $-$ & $-$ & $-$ & $\times$ & $-$ \\
$[4, 3, 5, 1, 2]$ & $s_{2}s_{3}s_{4}s_{1}s_{2}s_{3}s_{1}$ & $\times$ & $\times$ & $\times$ & $\times$ & $-$ & $-$ & $\times$ & $-$ \\
$[4, 3, 5, 2, 1]$ & $s_{1}s_{2}s_{3}s_{4}s_{1}s_{2}s_{3}s_{1}$ & $\times$ & $\times$ & $-$ & $\times$ & $-$ & $-$ & $\times$ & $-$ \\
$[4, 5, 1, 2, 3]$ & $s_{3}s_{4}s_{2}s_{3}s_{1}s_{2}$ & $\times$ & $\times$ & $\times$ & $\times$ & $-$ & $\times$ & $\times$ & $-$ \\
$[4, 5, 1, 3, 2]$ & $s_{2}s_{3}s_{4}s_{2}s_{3}s_{1}s_{2}$ & $\times$ & $\times$ & $-$ & $\times$ & $-$ & $\times$ & $\times$ & $-$ \\
$[4, 5, 2, 1, 3]$ & $s_{3}s_{4}s_{1}s_{2}s_{3}s_{1}s_{2}$ & $\times$ & $\times$ & $\times$ & $\times$ & $-$ & $\times$ & $\times$ & $-$ \\
$[4, 5, 2, 3, 1]$ & $s_{1}s_{2}s_{3}s_{4}s_{2}s_{3}s_{1}s_{2}$ & $\times$ & $\times$ & $-$ & $\times$ & $-$ & $-$ & $\times$ & $-$ \\
$[4, 5, 3, 1, 2]$ & $s_{2}s_{3}s_{4}s_{1}s_{2}s_{3}s_{1}s_{2}$ & $\times$ & $\times$ & $\times$ & $\times$ & $-$ & $-$ & $\times$ & $-$ \\
$[4, 5, 3, 2, 1]$ & $s_{1}s_{2}s_{3}s_{4}s_{1}s_{2}s_{3}s_{1}s_{2}$ & $\times$ & $\times$ & $-$ & $\times$ & $-$ & $-$ & $\times$ & $-$ \\
$[5, 1, 2, 3, 4]$ & $s_{4}s_{3}s_{2}s_{1}$ & $-$ & $-$ & $-$ & $-$ & $-$ & $-$ & $-$ & $-$ \\
$[5, 1, 2, 4, 3]$ & $s_{3}s_{4}s_{3}s_{2}s_{1}$ & $-$ & $-$ & $-$ & $-$ & $-$ & $-$ & $-$ & $-$ \\
$[5, 1, 3, 2, 4]$ & $s_{4}s_{2}s_{3}s_{2}s_{1}$ & $-$ & $-$ & $-$ & $-$ & $-$ & $-$ & $-$ & $-$ \\
$[5, 1, 3, 4, 2]$ & $s_{2}s_{3}s_{4}s_{3}s_{2}s_{1}$ & $\times$ & $\times$ & $-$ & $-$ & $\times$ & $-$ & $-$ & $-$ \\
$[5, 1, 4, 2, 3]$ & $s_{3}s_{4}s_{2}s_{3}s_{2}s_{1}$ & $\times$ & $-$ & $\times$ & $-$ & $-$ & $\times$ & $-$ & $-$ \\
$[5, 1, 4, 3, 2]$ & $s_{2}s_{3}s_{4}s_{2}s_{3}s_{2}s_{1}$ & $-$ & $-$ & $-$ & $-$ & $-$ & $-$ & $-$ & $-$ \\
$[5, 2, 1, 3, 4]$ & $s_{4}s_{3}s_{1}s_{2}s_{1}$ & $-$ & $-$ & $-$ & $-$ & $-$ & $-$ & $-$ & $-$ \\
$[5, 2, 1, 4, 3]$ & $s_{3}s_{4}s_{3}s_{1}s_{2}s_{1}$ & $-$ & $-$ & $-$ & $-$ & $-$ & $-$ & $-$ & $-$ \\
$[5, 2, 3, 1, 4]$ & $s_{4}s_{1}s_{2}s_{3}s_{2}s_{1}$ & $\times$ & $\times$ & $-$ & $-$ & $\times$ & $-$ & $-$ & $\times$ \\
$[5, 2, 3, 4, 1]$ & $s_{1}s_{2}s_{3}s_{4}s_{3}s_{2}s_{1}$ & $\times$ & $\times$ & $-$ & $-$ & $\times$ & $-$ & $-$ & $\times$ \\
$[5, 2, 4, 1, 3]$ & $s_{3}s_{4}s_{1}s_{2}s_{3}s_{2}s_{1}$ & $\times$ & $\times$ & $\times$ & $-$ & $\times$ & $\times$ & $-$ & $\times$ \\
$[5, 2, 4, 3, 1]$ & $s_{1}s_{2}s_{3}s_{4}s_{2}s_{3}s_{2}s_{1}$ & $\times$ & $\times$ & $-$ & $-$ & $\times$ & $-$ & $-$ & $\times$ \\
$[5, 3, 1, 2, 4]$ & $s_{4}s_{2}s_{3}s_{1}s_{2}s_{1}$ & $\times$ & $-$ & $\times$ & $-$ & $-$ & $\times$ & $-$ & $-$ \\
$[5, 3, 1, 4, 2]$ & $s_{2}s_{3}s_{4}s_{3}s_{1}s_{2}s_{1}$ & $\times$ & $-$ & $-$ & $-$ & $\times$ & $\times$ & $-$ & $-$ \\
$[5, 3, 2, 1, 4]$ & $s_{4}s_{1}s_{2}s_{3}s_{1}s_{2}s_{1}$ & $-$ & $-$ & $-$ & $-$ & $-$ & $-$ & $-$ & $-$ \\
$[5, 3, 2, 4, 1]$ & $s_{1}s_{2}s_{3}s_{4}s_{3}s_{1}s_{2}s_{1}$ & $\times$ & $-$ & $-$ & $-$ & $-$ & $-$ & $-$ & $\times$ \\
$[5, 3, 4, 1, 2]$ & $s_{2}s_{3}s_{4}s_{1}s_{2}s_{3}s_{2}s_{1}$ & $\times$ & $\times$ & $\times$ & $\times$ & $-$ & $-$ & $-$ & $\times$ \\
$[5, 3, 4, 2, 1]$ & $s_{1}s_{2}s_{3}s_{4}s_{1}s_{2}s_{3}s_{2}s_{1}$ & $\times$ & $\times$ & $-$ & $\times$ & $-$ & $-$ & $-$ & $\times$ \\
$[5, 4, 1, 2, 3]$ & $s_{3}s_{4}s_{2}s_{3}s_{1}s_{2}s_{1}$ & $\times$ & $-$ & $\times$ & $-$ & $-$ & $\times$ & $-$ & $-$ \\
$[5, 4, 1, 3, 2]$ & $s_{2}s_{3}s_{4}s_{2}s_{3}s_{1}s_{2}s_{1}$ & $\times$ & $-$ & $-$ & $-$ & $-$ & $\times$ & $-$ & $-$ \\
$[5, 4, 2, 1, 3]$ & $s_{3}s_{4}s_{1}s_{2}s_{3}s_{1}s_{2}s_{1}$ & $\times$ & $-$ & $\times$ & $-$ & $-$ & $\times$ & $-$ & $-$ \\
$[5, 4, 2, 3, 1]$ & $s_{1}s_{2}s_{3}s_{4}s_{2}s_{3}s_{1}s_{2}s_{1}$ & $-$ & $-$ & $-$ & $-$ & $-$ & $-$ & $-$ & $-$ \\
$[5, 4, 3, 1, 2]$ & $s_{2}s_{3}s_{4}s_{1}s_{2}s_{3}s_{1}s_{2}s_{1}$ & $\times$ & $-$ & $\times$ & $-$ & $-$ & $-$ & $-$ & $-$ \\
$[5, 4, 3, 2, 1]$ & $s_{1}s_{2}s_{3}s_{4}s_{1}s_{2}s_{3}s_{1}s_{2}s_{1}$ & $-$ & $-$ & $-$ & $-$ & $-$ & $-$ & $-$ & $-$ \\
\hline
120 & & $85$ & $64$ & $22$ & $57$ & $36$ & $22$ & $65$ & $8$ \\
\caption{Initial ideals $\init_{\bf w}(\mathcal{I}_w)$ (see \S\ref{sec:ideals}) for $w\in S_5$ and which criteria {to detect initial monomials from Theorem~\ref{thm:alllemma} apply.}}\label{tab:S5}
\end{longtable}
}

\end{document}